\documentclass[oneside,english]{amsart}
\usepackage[T1]{fontenc}
\usepackage[utf8]{inputenc}
\usepackage{float}
\usepackage{amsthm}
\usepackage{amstext}
\usepackage{graphicx}
\usepackage{amssymb}
\usepackage{esint}

\makeatletter

\providecommand{\tabularnewline}{\\}

\numberwithin{equation}{section}
\numberwithin{figure}{section}
  \theoremstyle{plain}
  \newtheorem*{thm*}{Theorem}
\theoremstyle{plain}
\newtheorem{thm}{Theorem}[section]
  \theoremstyle{plain}
  \newtheorem{prop}[thm]{Proposition}
  \theoremstyle{plain}
  \newtheorem{lem}[thm]{Lemma}
  \theoremstyle{plain}
  \newtheorem{cor}[thm]{Corollary}
  \theoremstyle{remark}
  \newtheorem*{rem*}{Remark}
  \theoremstyle{definition}
  \newtheorem{defn}[thm]{Definition}
  \theoremstyle{plain}
  \newtheorem{conjecture}[thm]{Conjecture}
  \theoremstyle{remark}
  \newtheorem{rem}[thm]{Remark}

\usepackage{ae}
\usepackage{aecompl} 

\makeatother

\usepackage{babel}

\begin{document}

\title{Isoperimetric Problems in Sectors With Density}

\maketitle
\begin{center}
Alexander Díaz, Nate Harman, Sean Howe, David Thompson
\par\end{center}
\begin{abstract}
We consider the isoperimetric problem in planar sectors with density
$r^{p}$, and with density $a>1$ inside the unit disk and $1$ outside.
We characterize solutions as a function of sector angle. We also solve
the isoperimetric problem in $\mathbb{R}^{n}$ with density $r^{p},\; p<0$.

\tableofcontents{}
\end{abstract}

\section{\label{sec:Introduction}Introduction}

A \emph{density }on the plane is a function weighting both perimeter
and area, or, on a general manifold, a function weighting surface
area and volume. Manifolds with density occur broadly in mathematics:
from probability, where the classical example of Gauss space ($\mathbb{R}^{n}$
with density $ce^{-a^{2}r^{2}}$) plays an important role, to Riemannian
geometry, where they appear as quotient spaces of Riemannian manifolds,
and even to physics, where they describe spaces with different mediums.
They have recently been in the spotlight for their role in Perelman's
proof of the Poincaré Conjecture. (For general references on these
topics, see \cite{Morgan-Manifolds with Density} or better \cite[Ch. 18]{Morgan - GMT}.)

We study the isoperimetric problem in planar sectors with certain
densities. In this context, the isoperimetric problem seeks to enclose
prescribed (weighted) area with least (weighted) perimeter (not counting
the boundary of the sector). Solutions are known for very few surfaces
with densities (see Sect. \ref{sec:Isoperimetric-problems-on-Manifolds-with-Density}
below). Our major theorem after Dahlberg \emph{et al. }\cite[Thm. 3.16]{Dahlberg et al - Isoperimetric regions in planes with density r^p}
characterizes isoperimetric curves in a $\theta_{0}$-sector with
density $r^{p}$, $p>0$:
\begin{thm*}
(\ref{thm:Major theorem on MWD}) Given $p>0$, there exist $0<\theta_{1}<\theta_{2}<\infty$
such that in the $\theta_{0}$-sector with density $r^{p}$, isoperimetric
curves are (see Fig. \ref{fig:Possible minimizers for sector w/density r^p}):

1. for $0<\theta_{0}<\theta_{1}$, circular arcs about the origin,

2. for $\theta_{1}<\theta_{0}<\theta_{2}$, undularies,

3. for $\theta_{2}<\theta_{0}<\infty$, semicircles through the origin.
\end{thm*}
We give bounds on $\theta_{1}$ and $\theta_{2}$ in terms of $p$,
and a conjecture on their exact values. Section \ref{sec:CGCC-1}
gives further results on constant generalized curvature curves. Sectors
with density $r^{p}$ are related to $L^{p}$ spaces (see \emph{e.g.}
Cor. \ref{cor:L^p norm inequality}), have vanishing generalized Gauss
curvature \cite[Def. 5.1]{Corwin et al - Differential geometry of manifolds with density},
and have an interesting singularity at the origin where density vanishes.
Adams \emph{et al. }\cite{Adams et al - Isoperimetric Regions in Gauss Sectors}
previously studied sectors with Gaussian density.

In Theorems \ref{thm:thetalessthanpi}, \ref{thm:Pi<thetanot<aPI},
and \ref{thm:aPI<thetanot}, we provide a relatively minor result
after Ca$\tilde{\text{n}}$ete \emph{et al. }\cite[Thm. 3.20]{Canete et al - Some isoperimetric problems in planes with density}
characterizing isoperimetric curves in a $\theta_{0}$-sector with
density $a>1$ inside the unit disk and density $1$ outside the unit
disk. An interesting property of this problem is that it deals with
a noncontinuous density. There are five different kinds of isoperimetric
regions depending on $\theta_{0}$, $a$, and the prescribed area,
shown in Figure \ref{fig:Isoperimetric-sets-for-thetanot-sector}.

In Section \ref{sec:R^n w/ radial} we discuss basic results in $\mathbb{R}^{n}$
with radial density. We use a simple averaging technique of Carroll
\emph{et al. \cite[Prop. 4.3]{Carroll et al - Isoperimetric problem on planes with density}
}to give a short proof of a result of Betta \emph{et al.} (\cite[Thm. 4.3]{Betta et al}
and our Thm. \ref{thm:BettaEtAl-R^n}) and to solve the isoperimetric
problem in $\mathbb{R}^{n}$ with density $r^{p},\; p<0$ (Thm. \ref{Pro:Hypersphere-Minimize-p_LT_-n}). 

\begin{figure}
\includegraphics[angle=-90,origin=c,width=4in,height=2in]{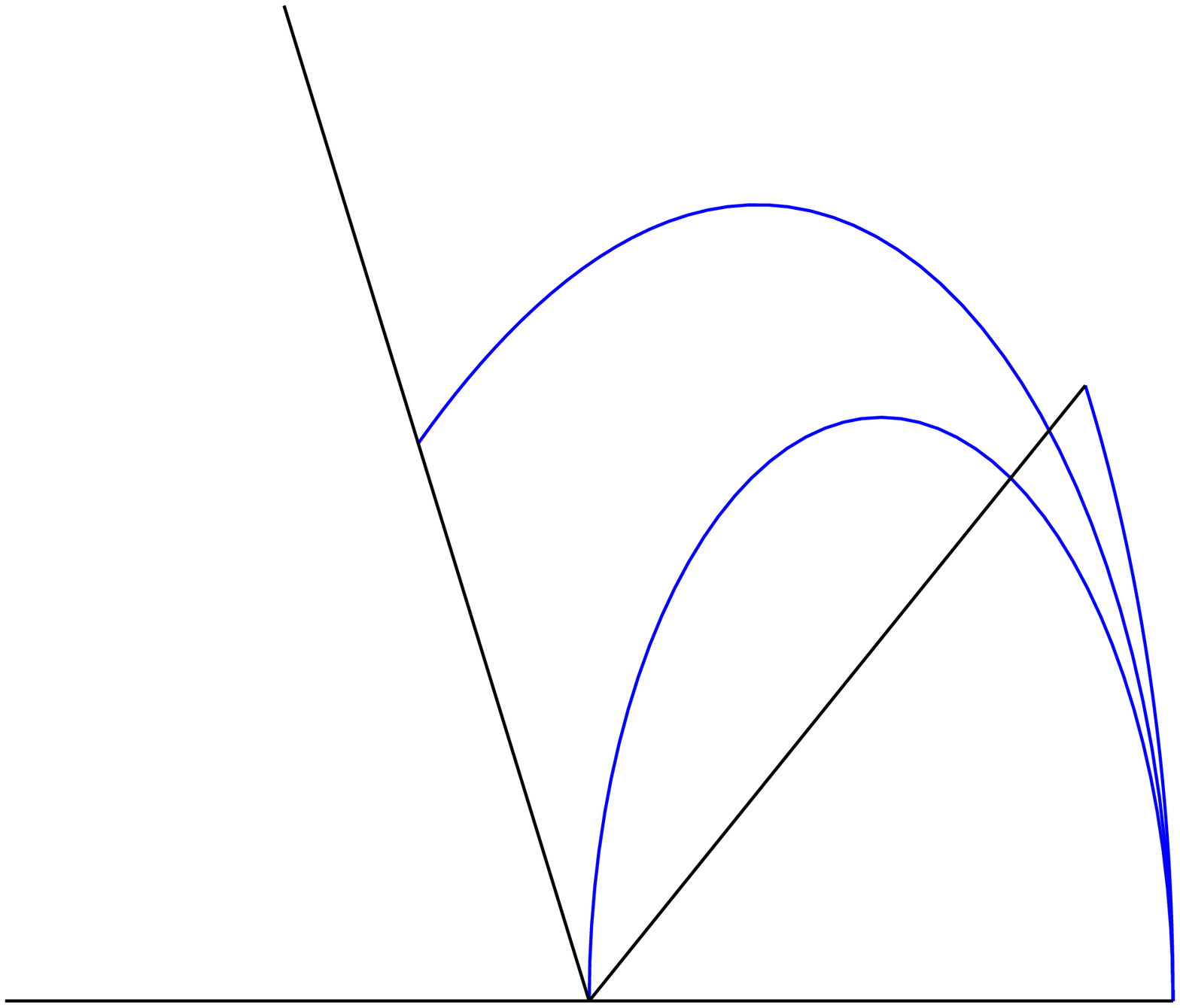}

\caption{\label{fig:Possible minimizers for sector w/density r^p}The three
isoperimetric curves for sectors with density $r^{p}$: a circular
arc about the origin for small sectors, an undulary for medium sectors,
and a semicircle through the origin for large sectors.}

\end{figure}

Due to length concerns, there are several points where we have chosen
to omit or shorten arguments that are either standard in the field
or similar to arguments given in earlier works. Among these are: the
existence results of Theorems \ref{The:Existence of minimizers in manifolds with area density satisfying certain properties},
the ball-density theorems of Section \ref{sec:Ball Density}, the
results on constant generalized curvature curves of Section \ref{sec:CGCC-1},
and the numeric results referred to in Subsection \ref{sub:The-Sector w/density r^p}
and seen in Figure \ref{fig:Transition-from-circular arc to semi-circle-1}.
For a more detailed account of our work, including full proofs and
expanded discussions of these topics, please see the 2009 SMALL Geometry
Group Report \cite{Report}.

\subsection{\label{sub:The-Sector w/density r^p}The Sector with Density $r^{p}$.}

In the plane with density $r^{p}$, Carroll \emph{et al}. \cite[Sect. 4]{Carroll et al - Isoperimetric problem on planes with density}
prove that for $p<-2$, isoperimetric curves are circles about the
origin bounding area on the outside and prove that for $-2\leq p<0$,
isoperimetric regions do not exist. Both of these results generalize
easily to sectors of arbitrary angle (Prop. \ref{pro:Non-existence-pin-2,0}
and the remark after Theorem \ref{thm:Major theorem on MWD}). Dahlberg
\emph{et al}. \cite[Thm. 3.16]{Dahlberg et al - Isoperimetric regions in planes with density r^p}
prove that for $p>0$, isoperimetric curves are circles through the
origin. By a simple symmetry argument (Prop. \ref{pro:cone-sector equivalence for iso. regions}),
isoperimetric circles about the origin and circles through the origin
in the plane correspond to isoperimetric circular arcs about the origin
and semicircles through the origin in a $\pi$-sector. In this paper,
we consider $\theta_{0}$-sectors for general $0<\theta_{0}<\infty$. 

For $p\in\left(-\infty,-2\right)\cup\left(0,\infty\right)$, existence
in the $\theta_{0}$-sector follows from standard compactness arguments
(Prop. \ref{prop:In the circular cone, isoperimetric exist for (inf,-2)U(o,inf)}).
Lemma \ref{lem:Minimizers are circles, semicircles, or unduloids}
limits the possibilities to circular arcs about the origin, semicircles
through the origin, and undularies (nonconstant positive polar graphs
with constant generalized curvature). Proposition \ref{pro:if circle minimizes, it minimizes for all theta less}
shows that if the circle is not uniquely isoperimetric for some angle
$\theta_{0}$, it is not isoperimetric for all $\theta>\theta_{0}$.
Proposition \ref{pro:The-isoperimetric-ratio is increasing for theta_0 < theta_2, constant after, semicircle uniquely}
shows that if the semicircle is isoperimetric for some angle $\theta_{0}$,
it is uniquely isoperimetric for all $\theta>\theta_{0}$. Therefore,
transitional angles $0\leq\theta_{1}\leq\theta_{2}\leq\infty$ exist.
Isoperimetric regions that depend on sector angle have been seen before,
as in the characterization by Lopez and Baker \cite[Thm. 6.1, Fig. 10]{Lopez - The double bubble problem on the cone}
of perimeter-minimizing double bubbles in the Euclidean cone of varying
angles, which is equivalent to the Euclidean sector. Theorem \ref{thm:Major theorem on MWD}
provides estimates on the values of $\theta_{1}$ and $\theta_{2}$.

We conjecture (Conj. \ref{con:Conjecture on major thm for MWD}) that
$\theta_{1}=\pi/\sqrt{p+1}$ and $\theta_{2}=\pi(p+2)/(2p+2)$. Proposition
\ref{pro:Circles are stable until pi/root p+1} proves that the circle
about the origin has positive second variation for all $\theta_{0}<\pi/\sqrt{p+1}$,
and Proposition \ref{pro:semi-circles don't minimize before pi/2*(p+2)/(p+1)}
proves the semicircle through the origin is not isoperimetric for
all $\theta_{0}<\pi(p+2)/(2p+2)$. Using the characterization of constant
generalized curvature curves in Section \ref{sec:CGCC-1}, we wrote
a computer program to predict isoperimetric curves in sectors with
density $r^{p}$, and the results of this program also support our
conjecture as can be seen for $p=1$ in Figure \ref{fig:Transition-from-circular arc to semi-circle-1}.

An easy symmetry argument (Prop. \ref{pro:cone-sector equivalence for iso. regions})
shows that the isoperimetric problem in the $\theta_{0}$-sector is
equivalent to the isoperimetric problem in the $2\theta_{0}$-cone,
a cone over $\mathbb{S}^{1}$. We further note that the isoperimetric
problem in the cone over $\mathbb{S}^{1}$ with density $r^{p}$ is
equivalent to the isoperimetric problem in the cone over the product
of $\mathbb{S}^{1}$ with a $p$-dimensional manifold $M$ among regions
symmetric under a group of isometries acting transitively on $M$.
This provides a classical interpretation of the problem, which we
use in Proposition \ref{pro:when p=00003D1, circles minimize until theta=00003D2}
to obtain an improved bound for $\theta_{1}$ in the $p=1$ case by
taking $M$ to be a rectangular two-torus.

Theorem \ref{thm:Polygon-r^p} classifies isoperimetric regions of
small area in planar polygons with density $r^{p}$. Proposition \ref{pro:MainTheoremPlanePerimeterDensity}
through Corollary \ref{cor:L^p norm inequality} reinterpret Theorem
\ref{thm:Major theorem on MWD} and Conjecture \ref{con:Conjecture on major thm for MWD}
in terms of the plane with differing area and perimeter densities
and in terms of analytic inequalities. 

\begin{figure}
\begin{tabular}{ccc}
\begin{tabular}{c}
\includegraphics[scale=0.5]{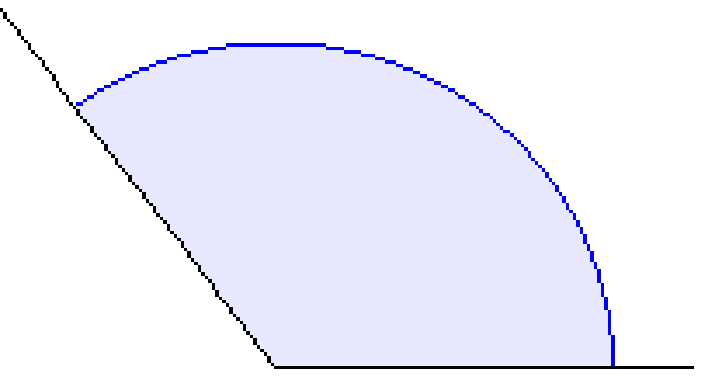}\tabularnewline
\noalign{\vskip0.05in}
$\dfrac{\pi}{\sqrt{2}}-.01$\tabularnewline
\end{tabular} & \begin{tabular}{c}
\includegraphics[scale=0.5]{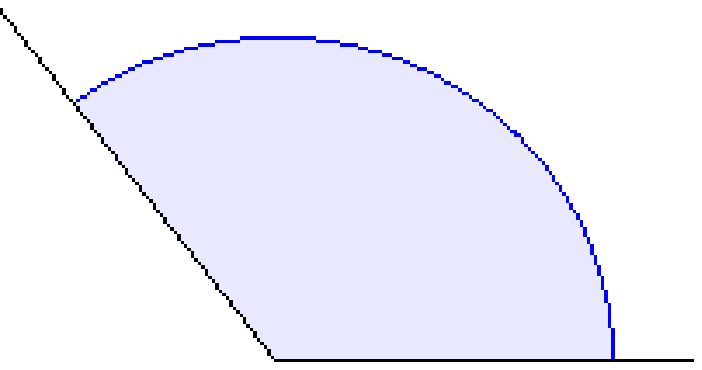}\tabularnewline
\noalign{\vskip0.05in}
$\dfrac{\pi}{\sqrt{2}}$\tabularnewline
\end{tabular} & \begin{tabular}{c}
\includegraphics[scale=0.5]{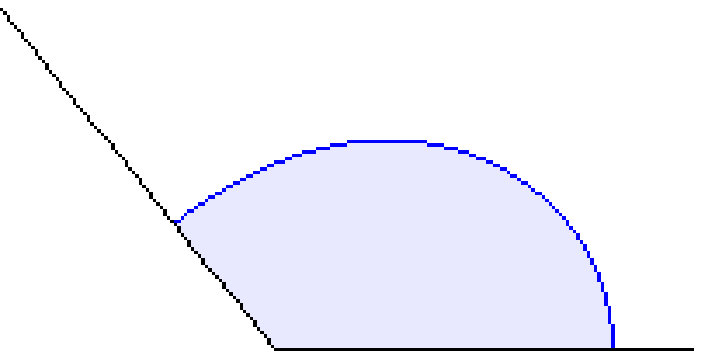}\tabularnewline
\noalign{\vskip0.05in}
$\dfrac{\pi}{\sqrt{2}}+.01$\tabularnewline
\end{tabular}\tabularnewline
\noalign{\vskip0.2in}
\noalign{\vskip0.05in}
\begin{tabular}{c}
\includegraphics[scale=0.5]{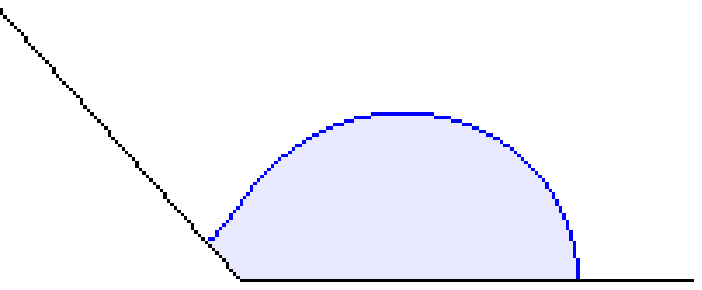}\tabularnewline
\noalign{\vskip0.05in}
$\dfrac{3\pi}{4}-.07$\tabularnewline
\end{tabular} & \begin{tabular}{c}
\includegraphics[scale=0.5]{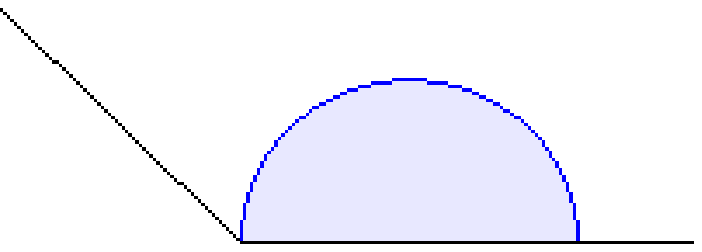}\tabularnewline
\noalign{\vskip0.05in}
$\dfrac{3\pi}{4}$\tabularnewline
\end{tabular} & \begin{tabular}{c}
\includegraphics[scale=0.5]{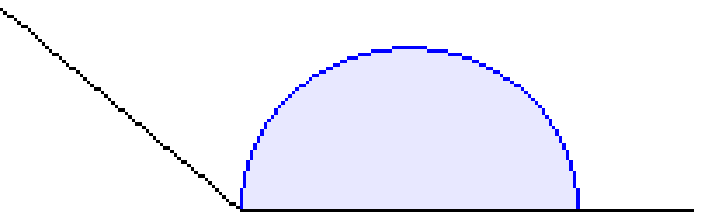}\tabularnewline
\noalign{\vskip0.05in}
$\dfrac{3\pi}{4}+.07$\tabularnewline
\end{tabular}\tabularnewline
\noalign{\vskip0.2in}
\end{tabular}\caption{\label{fig:Transition-from-circular arc to semi-circle-1}Theorem
\ref{thm:Major theorem on MWD} says that there are angles $\theta_{1}$
and $\theta_{2}$ such that isoperimetric curves are circular arcs
until $\theta_{1}$, undularies between $\theta_{1}$ and $\theta_{2}$,
and semicircles after $\theta_{2}$. For $p=1$, Conjecture \ref{con:Conjecture on major thm for MWD}
says that $\theta_{1}=\pi/\sqrt{2}$ and $\theta_{2}=3\pi/4$. Numerical
predictions of isoperimetric regions (above) agree with this conjecture. }

\end{figure}

\subsection{Constant Generalized Curvature Curves}

Section \ref{sec:CGCC-1} provides further details on constant generalized
curvature curves in the sector with density $r^{p}$, which are of
interest since isoperimetric curves must have constant generalized
curvature (see Sect. \ref{sec:Isoperimetric-problems-on-Manifolds-with-Density}).
Proposition \ref{pro:If period of CGCC increasing, then...-1} proves
Conjecture \ref{con:Conjecture on major thm for MWD} under the hypothesis
that undulary periods are bounded by the conjectured values of $\theta_{1}$
and $\theta_{2}$.

\subsection{The Sector with Disk Density}

Section \ref{sec:Ball Density} considers a sector of the plane with
density $a>1$ inside the unit disk and $1$ outside. Ca$\tilde{\text{n}}$ete
\emph{et al.} \cite[Sect. 3.3]{Canete et al - Some isoperimetric problems in planes with density}
consider this problem in the plane, which is equivalent to the $\pi$-sector.
Proposition \ref{pro:anyarea-isosets-are} gives the five possibilities
of Figure \ref{fig:Isoperimetric-sets-for-thetanot-sector}. Our Theorems
\ref{thm:thetalessthanpi}, \ref{thm:Pi<thetanot<aPI}, and \ref{thm:aPI<thetanot}
classify isoperimetric curves in a $\theta_{0}$-sector, depending
on $\theta_{0}$, density $a$, and area.

\subsection{$\mathbb{R}^{n}$ with Radial Density}

Section \ref{sec:R^n w/ radial} considers the isoperimetric problem
in $\mathbb{R}^{n}$ with radial density. We use spherical symmetrization
to reduce the problem to a two-dimensional isoperimetric problem in
a plane with density. Using ideas from Carroll \emph{et al.} \cite[Prop. 4.3]{Carroll et al - Isoperimetric problem on planes with density},
in Theorem \ref{thm:BettaEtAl-R^n} we give a new proof of a result
of Betta \emph{et al. \cite[Thm 4.3]{Betta et al}} on perimeter densities
in $\mathbb{R}^{n}$. We then use the same technique to prove that
hyperspheres about the origin are isoperimetric in $\mathbb{R}^{n}$
with density $r^{p},\quad p<-n$. In Proposition \ref{pro:Minimizers don't exist for -n<p<0 in R^n}
we provide a nonexistence result for $-n\leq p<0$, and we conclude
by conjecturing that hyperspheres through the origin are isoperimetric
for $p>0$ (Conj. \ref{con:R^n with density r^p conjecture}).

\subsection{Open Questions}

We provide here a concise list of open questions related to this paper:
\begin{enumerate}
\item How can one prove Conjecture \ref{con:Conjecture on major thm for MWD}
on the values of the transitional angles $\theta_{1}$ and $\theta_{2}$?
\item Could the values of $\theta_{1}$ and $\theta_{2}$ be proven numerically
for fixed $p$?
\item If a circular arc about the origin is isoperimetric in the $\theta_{0}$-sector
with density $r^{p}$, is it isoperimetric in the same $\theta_{0}$-sector
with density $r^{q}$, $q<p$?
\item Are circles about the origin isoperimetric in the Euclidean plane
with perimeter density $r^{p}$, $p\in(0,1)$? (See Rmk. after Conj.
\ref{con:Conjecture in plane with different densities}.)
\item In the plane with density $r^{p}$, do curves with constant generalized
curvature near that of the semicircle have half period $T\approx\pi\left(p+2\right)/\left(2p+2\right)$?
(See Conj. \ref{con:Conjecture on major thm for MWD} and Rmk. \ref{rem:CGCC remark}
for the corresponding result near the circular arc.) 
\item Are spheres through the origin isoperimetric in $\mathbb{R}^{n}$
with density $r^{p}$, $p>0$? (See Conj. \ref{con:R^n with density r^p conjecture}.)
\end{enumerate}

\subsection{Acknowledgments}

This paper is the work of the 2009 SMALL Geometry Group in Granada,
Spain. We would like to thank the National Science Foundation and
Williams College for funding both SMALL and our trip to Granada. We
would also like to thank the University of Granada's Department of
Geometry and Topology for their work and support. In particular, we
thank Rafael L$\acute{\text{o}}$pez for opening his home and his
office to us. We would like to thank Antonio Ca$\tilde{\mbox{\text{n}}}$ete,
Manuel Ritor$\acute{\text{e}}$, Antonio Ros, and Francisco L$\acute{\text{o}}$pez
for their comments. We thank C$\acute{\text{e}}$sar Rosales for arranging
our wonderful accommodations at Carmen de la Victoria. We thank organizers
Antonio Martínez, Jos$\acute{\text{e}}$ Antonio G$\acute{\text{a}}$lvez,
and Francisco Torralbo of the Escuela de An$\acute{\text{a}}$lisis
Geom$\acute{\text{e}}$trico. Robin Walters of the University of Chicago,
Leonard J. Schulman of CalTech, Casey Douglas of St. Mary’s College,
and Gary Lawlor of BYU all provided helpful insights during our research.
Finally, we would like to thank our research advisor Frank Morgan,
without whom this paper would never have been able to take shape.

\section{\label{sec:Isoperimetric-problems-on-Manifolds-with-Density}Isoperimetric
Problems in Manifolds with Density}

A \emph{density} on a Riemannian manifold is a nonnegative, lower
semicontinuous function $\Psi(x)$ weighting both volume and hypersurface
area. In terms of the underlying Riemannian volume $dV_{0}$ and area
$dA_{0}$, the weighted volume and area are given by $dV=\Psi dV_{0},$
$dA=\Psi dA_{0}$. Manifolds with densities arise naturally in geometry
as quotients of Riemannian manifolds, in physics as spaces with different
mediums, in probability as the famous Gauss space $\mathbb{R}^{n}$
with density $\Psi=ce^{-a^{2}r^{2}}$, and in a number of other places
as well (see \cite{Morgan-Manifolds with Density} or better \cite[Ch. 18]{Morgan - GMT}).

For a curve in a two-dimensional manifold $M$ with density $e^{\psi}$,
we define the generalized curvature $\lambda$ by \[
\lambda=\kappa-\frac{d\psi}{d\mathbf{n}},\]
where $\kappa$ is the usual curvature and $\mathbf{n}$ is the unit
normal vector. This is the correct generalization of curvature to
manifolds with density in that it provides a generalization of variational
formulae \cite[Prop 3.2]{Corwin et al - Differential geometry of manifolds with density}. 

The \emph{isoperimetric problem} on a two-dimensional manifold with
density seeks to enclose a given weighted area with the least weighted
perimeter. As in the Riemannian case, for a smooth density isoperimetric
curves have constant generalized curvature \cite[Prop 4.2]{Corwin et al - Differential geometry of manifolds with density}.
The solution to the isoperimetric problem is known only for a few
manifolds with density including Gauss space (see \cite[Ch. 18]{Morgan - GMT})
and the plane with a handful of different densities (see Betta \emph{et
al. \cite{Betta et al}}, Cañete \emph{et al.} \cite[Sect. 3]{Canete et al - Some isoperimetric problems in planes with density},
Dahlberg \emph{et al.} \cite[Thm 3.16]{Dahlberg et al - Isoperimetric regions in planes with density r^p},
Engelstein \emph{et al.} \cite[Cor. 4.9]{Engelstein et al - Isoperimetric problems on the sphere and on surfaces with density},
Rosales \emph{et al.} \cite[Thm. 5.2]{Rosales et al - On the Isoperimetric Problem in Euclidean Space with Density},
and Maurmann and Morgan \cite[Cor. 2.2]{Maurmann Morgan - Isoperimetric comparison theorems for manifolds with density}).
Existence and regularity are discussed in our report \cite[Sect. 3]{Report},
and the major results are given here.
\begin{prop}
\label{pro:Non-existence-pin-2,0}In the $\theta_{0}$-sector with
density $r^{p}$ for $p\in[-2,0)$, no isoperimetric regions exist. \end{prop}
\begin{proof}
Carroll \emph{et al.} \cite[Prop. 4.2]{Carroll et al - Isoperimetric problem on planes with density}
prove this in the plane by constructing curves of arbitrarily low
perimeter bounding any area. Their argument extends immediately to
the sector.
\end{proof}
The next theorem gives a general existence condition on isoperimetric
surfaces in manifolds with volume density.
\begin{thm}
\label{The:Existence of minimizers in manifolds with area density satisfying certain properties}Suppose
$M^{n}$ is a smooth connected possibly non-complete $n$-dimensional
Riemannian manifold with isoperimetric function $I$ such that $\lim_{A\rightarrow\infty}I(A)=\infty$.
Suppose furthermore that every closed geodesic ball of finite radius
in $M$ has finite volume and finite boundary area and there is some
$C\in\mathbb{Z}^{+}$ and $x_{0}\in M$ such that the complement of
any closed geodesic ball of finite radius about $x_{0}$ contains
finitely many connected components and at most $C$ unbounded connected
components. Then for standard boundary area and any lower semi-continuous
positive volume density $f$ such that 
\begin{enumerate}
\item for some $x_{0}\in M^{n}$ $\sup\{f(x)|\mathrm{dist}(x,x_{0})>R\}$
goes to $0$ as $R$ goes to $\infty$;
\item for some $\epsilon>0$,$\{x|B(x,\epsilon)\textrm{ is not complete}\}$
has finite weighted volume;
\end{enumerate}
isoperimetric regions exist for any positive volume less than the
volume of $M$.\end{thm}
\begin{proof}
Using standard compactness arguments from geometric measure theory
we see that it suffices to show that no weighted volume can escape
outside of an increasing sequence of compact sets whose union is $M^{n}$.
That no weighted volume can escape to points of noncompleteness is
immediate from the second condition on the density. To show that none
can escape to infinity we first isolate a finite number of {}``ends''
of the manifold and then show that in each of these no weighted volume
can disappear either because there is only finite weighted volume
there or because we can use the standard isoperimetric inequality
combined with the density approaching 0. 
\end{proof}
The following proposition establishes an equivalence between various
sectors with possibly differing perimeter and area densities.
\begin{prop}
\label{pro:Power change of coordinates}For any $n\in\mathbb{R}\backslash\{0\}$
the $\theta_{0}$-sector with perimeter density $r^{p}$ and area
density $r^{q}$ is equivalent to the $|n|\theta_{0}$-sector with
perimeter density $r^{[(p+1)/n]-1}$ and area density $r^{[(q+2)/n]-2}$.\end{prop}
\begin{proof}
Make the change of coordinates $w=z^{n}/n$.
\end{proof}
The next two propositions give the existence and regularity of isoperimetric
curves in the main cases that we examine in the rest of the paper. 
\begin{prop}
\label{prop:In the circular cone, isoperimetric exist for (inf,-2)U(o,inf)}In
the $\theta_{0}$-sector with density $r^{p}$ for $p\in(-\infty,-2)\cup(0,\infty)$
isoperimetric curves exist and are smooth curves with constant generalized
curvature.\end{prop}
\begin{proof}
By Proposition \ref{pro:Power change of coordinates} the $\theta_{0}$-sector
with density $r^{p}$ is equivalent to a $(p+1)\theta_{0}$-sector
with Euclidean perimeter and area density $r^{q}$ for $q\in(-2,0)$.
As we will see in Proposition \ref{pro:cone-sector equivalence for iso. regions},
the isoperimetric problem in the sector with radial density is equivalent
to the isoperimetric problem in a circular cone of twice the angle
with the same radial density. Furthermore it suffices to consider
the cone with the vertex deleted, which is a smooth manifold. Existence
then follows from Theorem \ref{The:Existence of minimizers in manifolds with area density satisfying certain properties},
and smoothness from Morgan \cite[Cor. 3.8, Sect. 3.10]{Morgan - Regularity of isoperimetric hypersurfaces in Riemannian manifolds}.
Constant generalized curvature follows from variational arguments. 
\end{proof}
\begin{figure}
\includegraphics{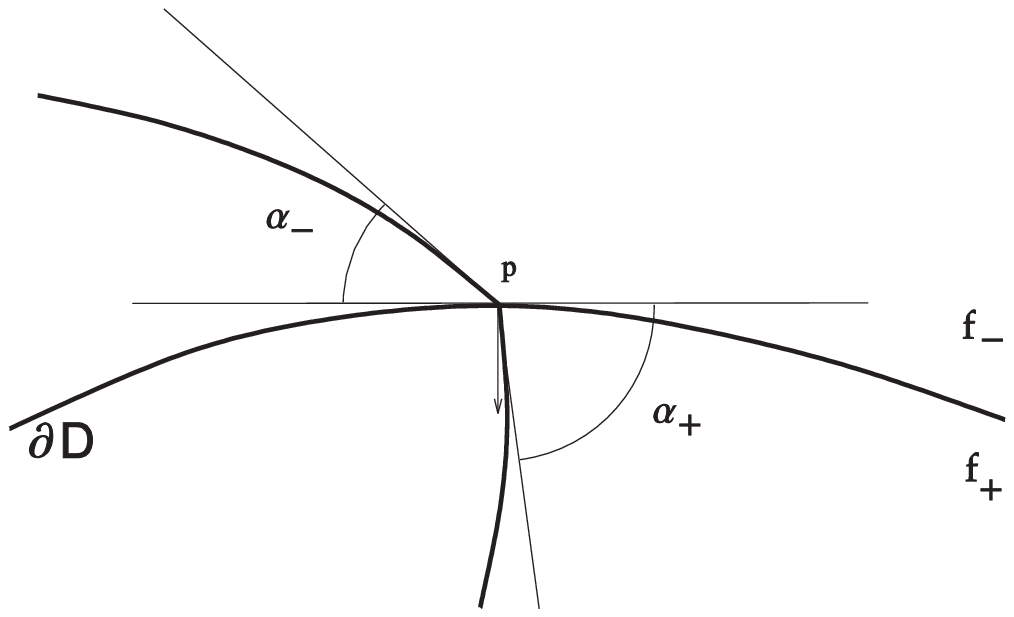}

\caption{\label{fig:Snell's law figure}The Snell refraction rule for curves
passing through the boundary of the unit disk. \cite[Fig. 1]{Canete et al - Some isoperimetric problems in planes with density},
used by permission, all rights reserved.}

\end{figure}

\begin{prop}
\label{pro:Isoperimetric exist in unit ball problem}In the $\theta_{0}$-sector
with density $a>1$ inside the unit disk $D$ and 1 outside, isoperimetric
curves exist for any given area. These isoperimetric curves are smooth
except at the boundary of $D$, where they obey the following Snell
refraction rule (see Fig. \ref{fig:Snell's law figure}):\[
\frac{\cos\alpha_{+}}{\cos\alpha_{-}}=\frac{1}{a},\]
where $a_{+}$ is the angle of intersection from inside of $D$ and
$\alpha_{-}$ is the angle of intersection from outside of $D$. \end{prop}
\begin{proof}
Cañete \emph{et al. }\cite[Thm. 3.18]{Canete et al - Some isoperimetric problems in planes with density}
prove existence for the plane. The same result and proof hold for
the $\theta_{0}$-sector\emph{.} The regularity and Snell refraction
rule follow from\emph{ }\cite[Prop. 2.14]{Canete et al - Some isoperimetric problems in planes with density}. 
\end{proof}

\section{Isoperimetric Regions in Sectors with Density\label{sec:Isoperimetric-Regions-in-Sectors-with-Density}}

We study properties of isoperimetric regions in planar sectors with
radial densities. Proposition \ref{pro:cone-sector equivalence for iso. regions}
shows there is a one-to-one correspondence between isoperimetric curves
in the $\theta_{0}$-sector and in the $2\theta_{0}$-cone, modulo
rotations. Propositions \ref{des:Proposition - isoperimetric curves meet boundary perpendicular-1}
and \ref{pro:Isoperimetric curves have strictly positive/negative derivative}
provide some regularity results.
\begin{lem}
\label{lem:there exist rays dividing iso. regions. in cone}Given
an isoperimetric region in the $2\theta_{0}$-cone with density $f(r)$,
there exist two rays from the origin separated by an angle of $\theta_{0}$
that divide both the area and perimeter of the region in half. \end{lem}
\begin{proof}
First we show that there are two rays that separate the area of the
region in half. Take any two rays from the origin separated by an
angle of $\theta_{0}$. Rotate them, keeping an angle of $\theta_{0}$
between them. The areas bounded by the rays and the boundary of the
isoperimetric region vary continuously as we rotate, and so does their
difference. Let the initial difference in areas be $\Delta A$. When
we rotate the rays by an angle of $\theta_{0}$, the difference in
areas is $-\Delta A$. So, there is a point where the difference is
$0$, implying that there are two rays separated by an angle $\theta_{0}$
that divide the area in half. If one side had less perimeter than
the other, we could reflect it to obtain a region with the same area
and less perimeter than our original region, violating the assumption
that it is isoperimetric.\end{proof}
\begin{prop}
\label{pro:cone-sector equivalence for iso. regions}An isoperimetric
region of area $2A$ in the cone of angle $2\theta_{0}$ with density
$f(r)$ has perimeter equal to the twice the perimeter of an isoperimetric
region of area $A$ in the sector of angle $\theta_{0}$ with density
$f(r)$. Indeed, the operation of doubling a sector to form a cone
provides a one-to-one correspondence between isoperimetric regions
in the sector and isoperimetric regions in the cone, modulo rotations.\end{prop}
\begin{proof}
Given an isoperimetric region in the sector, take its reflection into
the cone to obtain a region in the cone with twice the area and twice
the perimeter. This region must be isoperimetric for the cone, for
if there were a region in the cone with less perimeter for the same
area we could divide its area in half by two rays separated by an
angle $\theta_{0}$ as in Lemma \ref{lem:there exist rays dividing iso. regions. in cone},
take the side with at most half the perimeter to obtain a region in
the $\theta_{0}$-sector with the same area and less perimeter than
our original isoperimetric region. 

Conversely, given an isoperimetric region in the cone, divide its
area and perimeter in half by the two rays described in Lemma \ref{lem:there exist rays dividing iso. regions. in cone}.
Both regions must be isoperimetric in the sector, for if there were
a region with less perimeter for the same area, taking its double
would yield a region in the cone with less perimeter than our original
region for the same area.
\end{proof}
We shall find many occasions to use the following simple generalization
of a proposition of Dahlberg \emph{et al.}
\begin{prop}
\label{pro:CGCC symmetric about crit pts} Consider a circular cone
with smooth radial density $e^{\psi(r)}.$ A constant generalized
curvature curve is symmetric under reflection across every line through
the origin and a critical point of the radius function, in the sense
that if $\gamma$ is an arclength parameterization of the curve and
$s_{0}$ is a critical point of the curve's radius then for all $s$
such that $\gamma(s_{0}+s)$ and $\gamma(s_{0}-s)$ are defined we
have $\gamma(s_{0}+s)=R(\gamma(s_{0}-s))$ where $R$ is reflection
about that line. \end{prop}
\begin{proof}
As in \emph{\cite[Lem. 2.1]{Dahlberg et al - Isoperimetric regions in planes with density r^p}},
this follows from the uniqueness of solutions to ordinary differential
equations applied to the constant curvature equation. \end{proof}
\begin{prop}
\textbf{\label{des:Proposition - isoperimetric curves meet boundary perpendicular-1}
}In the $\theta_{0}$-sector with smooth density $f(r)$, isoperimetric
curves meet the boundary perpendicularly.\end{prop}
\begin{proof}
By Proposition \ref{pro:cone-sector equivalence for iso. regions}
there is a one-to-one correspondence between isoperimetric curves
in the sector and the cone, meaning the double of this curve in the
cone of angle $2\theta_{0}$ is isoperimetric and hence smooth (Prop.
\ref{prop:In the circular cone, isoperimetric exist for (inf,-2)U(o,inf)}).
But the double can only be smooth at the points corresponding to the
boundary points of the original curve if the intersection of the original
curve with the boundary is perpendicular. \end{proof}
\begin{prop}
\label{pro:Isoperimetric curves have strictly positive/negative derivative}In
the $\theta_{0}$-sector with smooth density $f(r)$, if isoperimetric
curves are nonconstant polar graphs (in the sense that any radial
ray intersects the curve exactly once), they do not contain a critical
point of radius on the interior.\end{prop}
\begin{proof}
Assume there is an isoperimetric curve with a critical point of radius
on the interior. Then by doubling as in Proposition \ref{pro:cone-sector equivalence for iso. regions}
we get an isoperimetric curve on the $2\theta_{0}$ cone, given in
polar coordinates by a function $r$. Since by \ref{des:Proposition - isoperimetric curves meet boundary perpendicular-1}
the initial curve met the boundary perpendicularly, $r$ has critical
points at $0$ and $\theta_{0}$, as well as the given critical point
between $0$ and $\theta_{0}$. Since constant generalized curvature
curves are symmetric under reflection across a line through the origin
and a critical point of $r$ (Prop. \ref{pro:CGCC symmetric about crit pts}),
we see that this curve has at least four critical points. By symmetry,
critical points must be strict extrema. Let $C$ be a circle about
the axis intersecting the curve in at least four points. $C$ divides
the curve into at least two regions above $C$ and two regions below
$C$. Interchanging one region above $C$ with a region below $C$
results in a region with the same perimeter and area whose boundary
is not smooth. Since isoperimetric curves must be smooth, $r$ cannot
be isoperimetric.\end{proof}
\begin{cor}
\label{cor:Isoperimetric curves are monotonic}If an isoperimetric
curve $r(\theta)$ in the $\theta_{0}$-sector with smooth density
is a nonconstant polar graph, $r$ must be strictly monotonic.\end{cor}
\begin{rem*}
Proposition \ref{pro:Isoperimetric curves have strictly positive/negative derivative}
and Corollary \ref{cor:Isoperimetric curves are monotonic} strengthen
some arguments of Adams \emph{et al.} \cite[Lem. 3.6]{Adams et al - Isoperimetric Regions in Gauss Sectors}.
\end{rem*}

\section{\label{sec:sector w/density r^p}The Isoperimetric Problem in Sectors
with Density $r^{p}$}

Our main result, Theorem \ref{thm:Major theorem on MWD}, characterizes
isoperimetric regions in planar sectors with density $r^{p}$, $p>0$.
The subsequent results consider reformulations of the problem in terms
of differing perimeter and area densities in the plane and in terms
of analytic inequalities. 
\begin{prop}
\label{pro:semi-circles minimize in half plane}In the half plane
with density $r^{p}$, $p>0$, semicircles through the origin are
the unique isoperimetric curves.\end{prop}
\begin{proof}
By Proposition \ref{pro:cone-sector equivalence for iso. regions},
there is a one-to-one correspondence between isoperimetric curves
in the $\theta_{0}$-sector and the $2\theta_{0}$-cone; in particular
there is a correspondence between isoperimetric curves in the half
plane and isoperimetric curves in the $2\pi$-cone, \emph{i.e.}, the
plane. Since circles through the origin are uniquely isoperimetric
in the plane (see Sect. \ref{sub:The-Sector w/density r^p}), semicircles
through the origin are uniquely isoperimetric in the half plane.\end{proof}
\begin{prop}
\label{pro:if circle minimizes, it minimizes for all theta less}For
density $r^{p}$, if the circular arc about the origin is not uniquely
isoperimetric in the $\theta_{0}$-sector, for all $\theta>\theta_{0}$
it is not isoperimetric.\end{prop}
\begin{proof}
Let $\gamma$ be a non-circular isoperimetric curve in the $\theta_{0}$-sector,
and let $C$ be a circular arc bounding the same area as $\gamma$.
For any angle $\theta>\theta_{0}$, transition to the $\theta$-sector
via the map $\alpha\rightarrow\alpha\theta/\theta_{0}$. This map
multiplies area by $\theta/\theta_{0}$, and scales tangential perimeter.
Therefore, if $\gamma$ had the same or less perimeter than $C$ in
the $\theta_{0}$-sector, its image under this map has less perimeter
than a circular arc about the origin in the $\theta$-sector.\end{proof}
\begin{rem*}
An equivalent statement of Proposition \ref{pro:if circle minimizes, it minimizes for all theta less}
is that if the circular arc about the origin is isoperimetric for
some sector angle $\theta_{0},$ then it is uniquely isoperimetric
for all $\theta<\theta_{0}.$\end{rem*}
\begin{prop}
\textbf{\label{pro:.isoperimetric curves are polar graphs containing origin}}
In the $\theta_{0}$-sector with density $r^{p}$, $p>0$, an isoperimetric
region contains the origin, and its boundary is a polar graph.\end{prop}
\begin{proof}
By Proposition \ref{pro:Power change of coordinates}, the $\theta_{0}$-sector
with density $r^{p}$ is equivalent to the $(p+1)\theta_{0}$ sector
with Euclidean perimeter and area density $r^{-q}$, $q=p/(p+1)$.
Since the properties in question are preserved by this change of coordinates,
we can work in this space. Since the area density is strictly decreasing
away from the origin, an isoperimetric region must contain the origin.
Any isoperimetric region is bounded by a smooth curve of constant
generalized curvature (Thm. \ref{prop:In the circular cone, isoperimetric exist for (inf,-2)U(o,inf)}),
and since here generalized curvature is just classical curvature divided
by the area density \cite[Def. 3.1, Prop. 3.2]{Carroll et al - Isoperimetric problem on planes with density},
the classical curvature does not change sign, and thus the curve is
convex. As the isoperimetric region contains the origin and its boundary
is convex, its boundary must be a polar graph. \end{proof}
\begin{prop}
\label{pro:more area/less perimeter argument works} In the $\theta_{0}$-sector
with density $r^{p}$, $p>0$, the least-perimeter ‘isoperimetric’
function $I(A)$ satisfies\[
I(A)=cA^{\frac{p+1}{p+2}}.\]
\end{prop}
\begin{proof}
Dahlberg \emph{et al.} prove this in the plane \emph{\cite[Lem. 3.7]{Dahlberg et al - Isoperimetric regions in planes with density r^p}},
and the same radial scaling argument works in the sector.\end{proof}
\begin{defn}
The \emph{isoperimetric ratio} of a $\theta_{0}$-sector with density
$r^{p}$ is the constant $c$ from Proposition \ref{pro:more area/less perimeter argument works}.
\end{defn}

\begin{defn}
An \emph{undulary }is a nonconstant positive polar graph with constant
generalized curvature.
\end{defn}
We note here a useful result from Dahlberg \emph{et al. }\cite[Prop. 2.11]{Dahlberg et al - Isoperimetric regions in planes with density r^p}:
In a planar domain with density $r^{p}$, $p>0$, if a constant generalized
curvature closed curve passes through the origin, it must be a circle.
\begin{lem}
\label{lem:Minimizers are circles, semicircles, or unduloids}In the
$\theta_{0}$-sector with density $r^{p}$, $p>0$, isoperimetric
curves are either circles about the origin, semicircles through the
origin, or undularies.\end{lem}
\begin{proof}
By Proposition \ref{prop:In the circular cone, isoperimetric exist for (inf,-2)U(o,inf)},
minimizers exist, and by Proposition \ref{pro:.isoperimetric curves are polar graphs containing origin},
they must be polar graphs with constant generalized curvature bounding
regions that contain the origin. If the minimizer has constant radius,
it is a circular arc. If the minimizer has nonconstant radius, it
either remains positive, in which case it is an undulary, or goes
through the origin. If the minimizer goes through the origin, it must
be part of a circle through the origin (\cite[Prop. 2.11]{Dahlberg et al - Isoperimetric regions in planes with density r^p}).
In order to meet the boundary perpendicularly, the curve must consist
of an integer number of semicircles. Suppose we have $n$ semicircles
each bounding a region of area $A_{i}$ with perimeter $P_{i}$. Since
all semicircles through the origin have the same ratio $P_{i}/A_{i}^{(p+1)/(p+2)}=c$,
we have\[
P=\sum P_{i}=c\sum A_{i}^{(p+1)/(p+2)}>c\left(\sum A_{i}\right)^{(p+1)/(p+2)}=P_{1},\]
where $P_{1}$ is the perimeter of a single semicircle through the
origin bounding area $A=\sum A_{i}$. Therefore, if a minimizer passes
through the origin, it is a single semicircle.\end{proof}
\begin{thm}
\label{thm:circular arcs minimize for pi/(p+1)}In the $\pi/(p+1)$-sector
with density $r^{p}$, $p>0$, isoperimetric curves are circular arcs.\end{thm}
\begin{proof}
We change coordinates as in Proposition \ref{pro:Power change of coordinates}
to the half plane with Euclidean perimeter. By Proposition \ref{pro:.isoperimetric curves are polar graphs containing origin},
an isoperimetric curve is a polar graph $r(\theta)$. Suppose that
it is not a circular arc. By Proposition \ref{pro:Isoperimetric curves have strictly positive/negative derivative},
$r(0)\neq r(\pi)$, and $r'(\theta)=0$ at $0$ and $\pi$, and nowhere
else. Reflect $r$ over the x-axis, obtaining a closed curve. By the
four-vertex theorem \cite{Osserman - the four or more vertex theorem},
this curve has at least four extrema of classical curvature. Generalized
curvature in Euclidean coordinates is classical curvature divided
by the area density \cite[Def. 3.1]{Carroll et al - Isoperimetric problem on planes with density}.
That is:\[
\kappa_{\varphi}=cr^{p/p+1}\kappa\]
for some $c>0$. At an extremum of classical curvature we see\[
0=\frac{d}{d\theta}\kappa_{\varphi}=\kappa'cr^{p/p+1}+c\frac{p}{p+1}r^{-1/p+1}r'\kappa=c\frac{p}{p+1}r^{-1/p+1}r'\kappa,\]
which implies either $r'=0$ or $\kappa=0$. However, if $\kappa=0$,
the curve is the geodesic, which is a straight line in Euclidean coordinates,
which cannot be isoperimetric. Therefore $r'=0$, meaning $r$ must
have a critical point other than $0$ and $\pi$, so it cannot be
isoperimetric. \end{proof}
\begin{rem*}
\label{rem:semi-circles minimize in euclidean half plane}After finding
this proof and examining the isoperimetric inequality in Proposition
\ref{prop:substituting alpha in isoperimetric inequality}, we came
across a more geometric proof. We show that in the Euclidean $\pi$-sector
with area density $cr^{-p/p+1}$, isoperimetric curves are circles
about the origin. When $p=0$, a semicircle about the origin is isoperimetric.
Now, for any $p>0$, suppose some region $R$ is isoperimetric. Take
a semicircle about the origin bounding the same Euclidean area; clearly
it will have less perimeter. However, it also has more weighted area,
because we have moved sections of $R$ that were further away from
the origin towards the origin. Since the area density is strictly
decreasing in $r$, we must have increased area. Therefore isoperimetric
curves are circles about the origin for the Euclidean $\pi$-sector
with area density $cr^{-p/p+1}$, implying that isoperimetric curves
are circular arcs in the $\pi/(p+1)$-sector with density $r^{p}$.\end{rem*}
\begin{prop}
\label{pro:when p=00003D1, circles minimize until theta=00003D2}
In the sector of $2$ radians with density $r$, isoperimetric curves
are circles about the origin.\end{prop}
\begin{proof}
Morgan \cite[Prop. 1]{Morgan-isoperimetric balls in cones over tori.}
proves that in cones over the square torus $\mathrm{\mathbb{T}}^{2}=\mathbb{S}^{1}(a)\times\mathbb{S}^{1}(a)$
isoperimetric regions are balls about the origin as long as $|\mathbb{T}^{2}|\leq|\mathbb{S}^{2}(1)|$.
We note that this proof still holds for rectangular tori $\mathbb{T}^{2}=\mathbb{S}^{1}(a)\times\mathbb{S}^{1}(b),$
$a\geq b$ so long as the ratio $a:b$ is at most $4:\pi$. Taking
the cone over the rectangular torus with side lengths 4 and $\pi$,
and modding out by the shorter copy of $\mathbb{S}^{1}$ we get the
cone over an angle 4 with density $\pi r$. Since isoperimetric regions
are balls about the origin in the original space, their images, circles
about the origin, are isoperimetric in this quotient space.\end{proof}
\begin{rem*}
Morgan and Ritor$\acute{\text{e}}$ \cite[Rmk. 3.10]{Morgan et al. - isoperimetric regions in cones}
ask whether $|M^{n}|\leq|\mathbb{S}^{n}(1)|$ is enough to imply that
balls about the origin are isoperimetric in the cone over $M$. Trying
to take the converse to the above argument we found an easy counterexample
to this question. Namely taking $M$ to be a rectangular torus of
area $4\pi$ with one very long direction and one very short direction,
we see that balls about the origin are not isoperimetric, as you can
do better with a circle through the origin across the short direction.

Along the same lines as Proposition \ref{pro:when p=00003D1, circles minimize until theta=00003D2}
one might hope to obtain bounds for other values of $p$ by examining
when balls about the origin are isoperimetric in cones over $\mathbb{S}^{1}(\theta)\times M^{p}$
for $M$ compact with a transitive isometry group and then modding
out by the symmetry group of $M$ to get the cone over $\mathbb{S}^{1}(\theta)$
with density proportional to $r^{p}$. In order to see when balls
about the origin are isoperimetric in the cone over $\mathbb{T}^{2}$,
Morgan uses the Ros product theorem with density \cite[Thm. 3.2]{Morgan-In Polytopes Small Balls about Some Vertex Minimize Perimeter },
which requires the knowledge of the isoperimetric profile of the link
(in his case $\mathbb{T}^{2}$). One such manifold of the form $\mathbb{S}^{1}\times M$
for which the isoperimetric problem is solved is $\mathbb{S}^{1}\times\mathbb{S}^{2}$
\cite[Thm. 4.3]{Pedrosa Ritore-isoperimetric domains in the Riemannian Product}.
The three types of isoperimetric regions in $\mathbb{S}^{1}\times\mathbb{S}^{2}$
are balls or complements of balls, tubular neighborhoods of $\mathbb{S}^{1}\times\{\mathrm{point}\}$,
or regions bounded by two totally geodesic copies of $\mathbb{S}^{2}$.
By far the most difficult of the three cases to deal with are the
balls, where unlike the other two cases we cannot explicitly compute
the volume and surface area. Fixing the volume of $\mathbb{S}^{1}\times\mathbb{S}^{2}$
as $2\pi^{2}=|\mathbb{S}^{3}|$ we can apply the same argument Morgan
uses in \cite[Lemma 2]{Morgan-isoperimetric balls in cones over tori.}
to show that if the sectional curvature of $\mathbb{S}^{1}\times\mathbb{S}^{2}$
is bounded above by 1 (by taking $\mathbb{S}^{2}$ large and $\mathbb{S}^{1}$
small), then balls in $\mathbb{S}^{1}\times\mathbb{S}^{2}$ do worse
than balls in $\mathbb{S}^{3}$, as desired, but unfortunately tubular
neighborhoods of $\mathbb{S}^{1}$ then sometimes beat balls in $\mathbb{S}^{3}$
for certain volumes, making it so we cannot apply the Ros product
theorem. So without a better way to deal with balls in $\mathbb{S}^{1}\times\mathbb{S}^{2}$
that does not depend on such a strong assumption on the sectional
curvature, this method does not work even for the $p=2$ case. Perhaps
there is another way to prove when balls about the vertex are isoperimetric
in the cone over $\mathbb{S}^{1}\times\mathbb{S}^{2}$ without using
the Ros product theorem.\end{rem*}
\begin{prop}
\label{pro:semi-circles don't minimize before pi/2*(p+2)/(p+1)}In
the $\theta_{0}$-sector with density $r^{p}$, semicircles through
the origin are not isoperimetric for $\theta_{0}<\pi\left(p+2\right)/\left(2p+2\right)$.\end{prop}
\begin{proof}
In Euclidean coordinates, semicircles through the origin terminate
at the angle $\theta=\left(\pi/2\right)\left(p+1\right)$. Since the
semicircle approaches this axis tangentially, for any $\theta_{0}<\left(\pi/2\right)(p+2)/(p+1)$,
there is a line normal to the boundary $\theta_{0}(p+1)$ in Euclidean
coordinates which intersects the semicircle at a single point $b$.
Replacing the segment of the semicircle from $b$ to the origin with
this line increases area while decreasing perimeter. Therefore semicircles
are not isoperimetric.\end{proof}
\begin{lem}
\label{lem:Inequality on extending increasing of isoperimetric ratio}If
the $\theta_{0}$-sector with density $r^{p}$ has isoperimetric ratio
$I_{0}$ and \[
I_{0}<\left(\frac{p+2}{p+1}\right)(p+2)^{(p+1)/(p+2)}\theta_{0}^{1/(p+2)},\]
then there exists an $\epsilon>0$ such that the isoperimetric ratio
of any $(\theta_{0}+t)$-sector with $0<t<\epsilon$ is greater than
or equal to the isoperimetric ratio of the $\theta_{0}$-sector, with
equality if and only if the semicircle is isoperimetric in the $\theta_{0}$-sector,
in which case the semicircle is uniquely isoperimetric in the $(\theta_{0}+t)$-sector. \end{lem}
\begin{proof}
Consider an isoperimetric curve $\gamma$ in the $(\theta_{0}+t)$-sector
bounding area 1. By reflection we can assume it is nonincreasing (Cor.
\ref{cor:Isoperimetric curves are monotonic}). We partition the $(\theta_{0}+t)$-sector
into a $\theta_{0}$-sector followed by a $t$-sector and we let $\alpha_{0}$
denote the area bounded by $\gamma$ in the $\theta_{0}$-sector and
$\alpha_{t}$ the area bounded by $\gamma$ in the $t-$sector. Note
$\alpha_{0}+\alpha_{t}=1$, and since $\gamma$ must be either an
undulary, a circular arc, or a semicircle, $\alpha_{t}=0$ if and
only if $\gamma$ is a semicircle. Then we can bound the isoperimetric
ratio $I_{t}$ of the $(\theta_{0}+t)$-sector by using the isoperimetric
ratios $I_{0}$ for the $\theta_{0}$-sector and $R_{t}$ for the
$t$-sector as follows: \[
I_{t}=P(\gamma)\geq I_{0}\alpha_{0}^{(p+1)/(p+2)}+R_{t}\alpha_{t}^{(p+1)/(p+2)}.\]
Since the radius of $\gamma$ is nonincreasing we know that the $\theta_{0}$-sector
contains at least its angular proportion of the area and thus $\alpha_{0}\geq\theta_{0}/(\theta_{0}+t)$.
We now substitute $\alpha_{t}=1-\alpha_{0}$ and look at the right
side of the inequality as a function of $\alpha_{0}$: \[
f_{t}(\alpha_{0})=I_{0}\alpha_{0}^{(p+1)/(p+2)}+R_{t}(1-\alpha_{0})^{(p+1)/(p+2)}.\]
This function is concave, so it attains its minimum at an endpoint.
Since $\alpha_{0}$ is bounded between $\theta_{0}/(\theta_{0}+t)$
and $1$, we see that $I_{t}$ is greater than or equal to the minimum
of $f_{t}(\theta_{0}/(\theta_{0}+t))$ and $f_{t}(1)$. Since $f_{t}(1)=I_{0}$,
we want to show that there is some $\epsilon$ such that $t<\epsilon$
implies $f_{t}(\theta_{0}/(\theta_{0}+t))>f(1)=I_{0}$. So we define
a new function \[
g(t)=f_{t}(\theta_{0}/(\theta_{0}+t))-I_{0}.\]
We want to show that there is some positive neighborhood of $0$ where
$g(t)>0$. For $t<\pi/(p+1)$, isoperimetric regions are circular
arcs, so for $t\in(0,\delta)$, $R_{t}=(p+2)^{(p+1)/(p+2)}t^{1/p+2}$
and $g$ is differentiable. Furthermore, $\lim_{t\rightarrow0}g(t)=0$
and so it suffices to prove that $\lim_{t\rightarrow0}g'(t)>0$. We
calculate \begin{eqnarray*}
\lim_{t\rightarrow0}g'(t) & = & -I_{0}\left(\dfrac{p+1}{p+2}\right)\theta_{0}^{-1}+\\
 &  & (p+2)^{(p+1)/(p+2)}\theta_{0}^{-(p+1)/(p+2)}\end{eqnarray*}
and deduce that this is greater than $0$ if and only if\[
I_{0}<\left(\frac{p+2}{p+1}\right)(p+2)^{(p+1)/(p+2)}\theta_{0}^{1/(p+2)}.\]

Thus, we get a neighborhood where $g(t)$ is positive. Now, recall
$I_{t}\geq f_{t}(\alpha_{0})$ where $\alpha_{0}$ is the area bounded
by the isoperimetric curve $\gamma$ in the $\theta_{0}$-sector,
and $\alpha_{0}\in[\theta_{0}/(\theta_{0}+t),1]$. Since $g(t)$ is
positive, $f_{t}(\alpha_{0})\geq f_{t}(1)=I_{0}$, with equality holding
if and only if $\alpha_{0}=1,$ that is, only if $\gamma$ is a semicircle.
If $\gamma$ is a semicircle and the semicircle is not isoperimetric
in the $\theta_{0}$-sector then we must have $I_{t}>I_{0}$ since
otherwise the semicircle would be isoperimetric in the $\theta_{0}$-sector.
(Note that by Proposition \ref{pro:semi-circles don't minimize before pi/2*(p+2)/(p+1)}
if $\theta_{0}<\pi/2$ then we can decrease the size of our neighborhood
so that the semicircle is not isoperimetric anywhere in it and thus
we do not have to worry about non-existence of the semicircle here).
On the other hand if the semi-circle is isoperimetric in the $\theta_{0}$-sector
then we see that it is uniquely isoperimetric in the $(\theta_{0}+t)$-sector
since any other curve will have $\alpha_{0}<1$. \end{proof}
\begin{prop}
\label{pro:The-isoperimetric-ratio is increasing for theta_0 < theta_2, constant after, semicircle uniquely}For
fixed $p>0$, the isoperimetric ratio of the $\theta_{0}$-sectors
with density $r^{p}$ is an increasing function of $\theta_{0}$ for
$\theta_{0}<\theta_{2}$, where $\theta_{2}$ is the angle at which
the semicircle is first isoperimetric. On $[\theta_{2},\infty)$ the
isoperimetric ratio is constant and the semicircle is uniquely isoperimetric
on $(\theta_{2},\infty)$. \end{prop}
\begin{proof}
For a circle about the origin $P=\theta_{0}^{1/p+2}(p+2)^{(p+1)/(p+2)}A^{(p+1)/(p+2)}$,
so we see that for all $p>0$ the isoperimetric ratio $I_{0}$ satisfies
the conditions of Lemma \ref{lem:Inequality on extending increasing of isoperimetric ratio}
for every $\theta_{0}$. By Proposition \ref{pro:semi-circles minimize in half plane},
the semicircle is isoperimetric in the $\pi$-sector, and since the
isoperimetric ratio is continuous in $\theta_{0}$ there is some minimum
$\theta_{2}$ where the semicircle is isoperimetric. By Lemma \ref{lem:Inequality on extending increasing of isoperimetric ratio},
the semicircle is isoperimetric on $(\theta_{2},\infty)$. Thus, the
isoperimetric ratio is the same for all of these sectors, and so the
functions $f_{t}$ from the proof of the lemma are the same for all
$\theta_{0}\geq\theta_{2}$. Thus we see that we can pick the $\epsilon$
given by the lemma uniformly so that for all $\theta_{0}\geq\theta_{2}$,
the semicircle uniquely minimizes on $(\theta_{0},\theta_{0}+\epsilon)$,
giving us that the semicircle uniquely minimizes on $(\theta_{2},\infty)$.
On the other hand, since the isoperimetric ratio is continuous in
$\theta_{0}$, it is clear from Lemma \ref{lem:Inequality on extending increasing of isoperimetric ratio}
that it is strictly increasing on $\theta_{0}<\theta_{2}$.\end{proof}
\begin{cor}
\label{cor:The-semicircle-minimizes in half parking garage}The semicircle
is isoperimetric in the {}``half-infinite parking garage'' $\{(\theta,\text{ }r)|\theta\geq0,\text{ }r>0\}$
with density $r^{p}$.\end{cor}
\begin{proof}
Suppose $\gamma$ is isoperimetric in the half-infinite parking garage.
For any $\theta_{0}\geq\pi$ the restriction of $\gamma$ to the $\theta_{0}$
sector has isoperimetric ratio greater than or equal to that of the
semicircle. Since $\lim_{\theta_{0}\rightarrow\infty}P(\gamma|_{\theta_{0}})=P(\gamma)$
and $\lim_{\theta_{0}\rightarrow\infty}A(\gamma|_{\theta_{0}})=A(\gamma)$
the limit of the isoperimetric ratios of $\gamma|_{\theta_{0}}$ is
the isoperimetric ratio of $\gamma$ and so we see it is also greater
than or equal to that of the semicircle. Since the semicircle exists
in the half-infinite parking garage we are done. \end{proof}
\begin{rem*}
Studying the isoperimetric ratio turns out to be an extremely useful
tool in determining the behavior of semicircles for $\theta_{0}>\pi$.
However, it is not the only such tool. Here we give an entirely different
proof that semicircles are isoperimetric for all $n\pi$-sectors.\end{rem*}
\begin{prop}
\label{pro: minimizers minimize in sector w/multiplicity}In the $\theta_{0}$-sector
with density $r^{p}$, $p>0$, even when allowing multiplicity greater
than one, isoperimetric regions will not have multiplicity greater
than one.\end{prop}
\begin{proof}
A region $R$ with multiplicity may be decomposed as a sum of nested
regions $R_{j}$ with perimeter and area \cite[Fig. 10.1.1]{Morgan - GMT}:\[
P(R)=\sum P(R_{j}),\]
\[
A(R)=\sum A(R_{j}).\]
Let $R'$ be an isoperimetric region of multiplicity one and the same
area as $R$. Since isoperimetric regions remain isoperimetric under
scaling, for each region $P_{j}\geq cA_{j}^{\left(p+1\right)/\left(p+2\right)}$,
where $c=P_{R'}/A_{R'}^{(p+1)/(p+2)}$ (Prop. \ref{pro:more area/less perimeter argument works}).
By concavity\[
P(R)=\sum P(R_{j})\geq c\sum\left(A(R_{j})^{\frac{p+1}{p+2}}\right)\geq c\left(\sum A(R_{j})\right)^{\frac{p+1}{p+2}}=P(R'),\]
with equality only if $R$ has multiplicity one. Therefore no isoperimetric
region can have multiplicity greater than one.\end{proof}
\begin{rem*}
For $p<-2$ the isoperimetric function $I(A)=cA^{(p+1)/(p+2)}$ is
now convex. Therefore, regions with multiplicity greater than one
can do arbitrarily better than regions with multiplicity one.\end{rem*}
\begin{cor}
\label{cor:semi-circles minimize for n*pi sector}In the $n\pi$-sector
$(n\in\mathbb{Z})$, semicircles through the origin are uniquely isoperimetric.\end{cor}
\begin{proof}
Assume there is an isoperimetric curve $r(\theta)$ bounding a region
$R$ which is not the semicircle. Consider $R$ as a region with multiplicity
in the half plane by taking $r(\theta)\rightarrow r(\theta\mbox{\text{ }}mod\,\pi)$.
Since semicircles are uniquely isoperimetric in the half plane, by
Proposition \ref{pro: minimizers minimize in sector w/multiplicity},
$r$ cannot be isoperimetric. This implies that $r$ could not have
been isoperimetric in the $n\pi$-sector.\end{proof}
\begin{prop}
\label{pro:Circles are stable until pi/root p+1}In the $\theta_{0}$-sector
with density $r^{p}$, $p>-1$, circular arcs about the origin have
positive second variation if and only if $\theta_{0}<\pi/\sqrt{p+1}$.
When $p<-1$, circular arcs about the origin always have positive
second variation.\end{prop}
\begin{proof}
By Proposition \ref{pro:cone-sector equivalence for iso. regions}
we think of the $\theta_{0}$-sector as the cone of angle $2\theta_{0}$.
A circle of radius $r$ in the $\theta_{0}$-sector corresponds with
a circle about the axis with radius $r\theta_{0}/\pi$, giving the
cone the metric $ds^{2}=dr^{2}+(r\theta_{0}/\pi)^{2}d\theta^{2}$.
For a smooth Riemannian disk of revolution with metric $ds^{2}=dr^{2}+f(r)^{2}d\theta^{2}$
and density $e^{\psi(r)}$, circles of revolution at distance $r$
have positive second variation if and only if $Q(r)=f'(r)^{2}-f(r)f''(r)-f(r)^{2}\psi''(r)<1$
\cite[Thm. 6.3]{Engelstein et al - Isoperimetric problems on the sphere and on surfaces with density}.
This corresponds to $(\theta_{0}/\pi)^{2}+p(\theta_{0}/\pi)^{2}<1$,
which, for $p<-1$, always holds. When $p>-1$, the condition becomes
$\theta_{0}<\pi/\sqrt{p+1}$, as desired.
\end{proof}
The following theorem is the main result of this paper.
\begin{thm}
\label{thm:Major theorem on MWD}Given $p>0$, there exist $0<\theta_{1}<\theta_{2}<\infty$
such that in the $\theta_{0}$-sector with density $r^{p}$, isoperimetric
curves are (see Fig. \ref{fig:Possible minimizers for sector w/density r^p}):

1. for $0<\theta_{0}<\theta_{1}$, circular arcs about the origin,

2. for $\theta_{1}<\theta_{0}<\theta_{2}$ , undularies,

3. for $\theta_{2}<\theta_{0}<\infty$, semicircles through the origin.

Moreover, \[
\pi/(p+1)\leq\theta_{1}\leq\pi/\sqrt{p+1},\]
\[
\pi(p+2)/(2p+2)\leq\theta_{2}\leq\pi.\]
When $p=1$, $\theta_{1}\geq2>\pi/2\approx1.57$.\end{thm}
\begin{rem*}
For $p<-2$, circular arcs about the origin bounding area away from
the origin are isoperimetric for all sectors, and for $-2\leq p<0$
isoperimetric regions do not exist. The proofs given by Carroll \emph{et
al. }\cite[Prop. 4.3]{Carroll et al - Isoperimetric problem on planes with density}
generalize immediately from the plane to the sector. In Section \ref{sec:R^n w/ radial}
we give a generalization to $\mathbb{R}^{n},\; n\geq2$ of these statements
about $p<0$.\end{rem*}
\begin{proof}
By Lemma \ref{lem:Minimizers are circles, semicircles, or unduloids},
minimizers exist and must be circles, undularies, or semicircles.
As $\theta$ increases, if the circle is not minimizing, it remains
not minimizing (Prop. \ref{pro:if circle minimizes, it minimizes for all theta less}).
If the semicircle is minimizing, it remains uniquely minimizing (Prop.
\ref{pro:The-isoperimetric-ratio is increasing for theta_0 < theta_2, constant after, semicircle uniquely}).
Therefore transitional angles $0\leq\theta_{1}\leq\theta_{2}\leq\infty$
exist. Strict inequalities are trivial consequences of the following
inequalities:

To prove $\theta_{1}\geq\pi/(p+1)$, note that circular arcs are the
unique minimizers for $\theta_{0}=\pi/(p+1)$ (Thm. \ref{thm:circular arcs minimize for pi/(p+1)}), 

To prove $\theta_{1}\leq\pi/\sqrt{p+1}$, recall that circular arcs
do not have nonnegative second variation for $\theta_{0}>\pi/\sqrt{p+1}$
(Prop. \ref{pro:Circles are stable until pi/root p+1}). To prove
$\theta_{2}\geq\pi(p+2)/(2p+2)$, recall that semicircles cannot minimize
for $\theta_{0}<\pi(p+2)/(2p+2)$ (Prop. \ref{pro:semi-circles don't minimize before pi/2*(p+2)/(p+1)}).
To prove $\theta_{2}\leq\pi$, recall that semicircles minimize for
$\theta_{0}=\pi$ (Prop. \ref{pro:semi-circles minimize in half plane}).
Since $\pi/\sqrt{p+1}<\pi(p+2)/(2p+2)$ for all $p>0$, we have $\theta_{1}<\theta_{2}$.
For $\theta_{1}<\theta_{0}<\theta_{2}$, neither the circle nor the
semicircle minimizes, so minimizers are undularies. Finally, when
$p=1$, circles minimize for $\theta_{0}=2$ (Prop. \ref{pro:when p=00003D1, circles minimize until theta=00003D2}).
\end{proof}
We conjecture that the circle is isoperimetric as long as it has nonnegative
second variation, and that the semicircle is isoperimetric for all
angles greater than $\pi\left(p+2\right)/\left(2p+2\right)$.
\begin{conjecture}
\label{con:Conjecture on major thm for MWD}In Theorem \ref{thm:Major theorem on MWD},
the transitional angles $\theta_{1}$, $\theta_{2}$ are given by
$\theta_{1}=\pi/\sqrt{p+1}$ and $\theta_{2}=\pi(p+2)/(2p+2)$. \end{conjecture}
\begin{rem*}
This conjecture is supported by numeric evidence as in Figure \ref{fig:Transition-from-circular arc to semi-circle-1}. 
\end{rem*}

\begin{rem*}
Our isoperimetric undularies give explicit examples of the abstract
existence result of Rosales \emph{et al.} \cite[Cor. 3.13]{Rosales et al - On the Isoperimetric Problem in Euclidean Space with Density}
of isoperimetric regions not bounded by lines or circular arcs.
\end{rem*}
One potential avenue for proving this conjecture is discussed in Proposition
\ref{pro:If period of CGCC increasing, then...-1}. We also believe
the transition between the circle and the semicircle is parametrized
smoothly by curvature, which is discussed in Section \ref{sec:CGCC-1}.

Theorem \ref{thm:Polygon-r^p} will allow us to apply Theorem \ref{thm:Major theorem on MWD}
to the problem of classifying isoperimetric regions of small area
in planar polygons with density $r^{p}$, as suggested to us by Antonio
Ca$\tilde{\text{n}}$ete. It is similar to an isoperimetric theorem
of Morgan on regions of small area in polytopes \cite[Thm. 3.8]{Morgan-In Polytopes Small Balls about Some Vertex Minimize Perimeter }.
First we need a lemma:
\begin{lem}
\label{lem:Polygon-Ineq-With-Density}Consider a polygon in the plane
with density $f$. Then, for any region $R$ inside the polygon where
the density is bounded away from 0 and $\infty$, there exists a constant
$c>0$ such that any sub-region $S$ bounding Euclidean area less
than half the Euclidean area of the polygon has weighted area $A$
and weighted perimeter $P$ satisfying\[
P\geq cA^{1/2}\]
\end{lem}
\begin{proof}
Let $P'$ and $A'$ denote the Euclidean perimeter and area of $S$
and let $m$ be the minimum of the density in $R$ and $M$ the maximum
of the density in $R$. Then $P\geq m\cdot P'$ and $A\leq M\cdot A'$.
By a standard result $P'\geq cA'^{1/2}$ for some constant $c$ (see,
e.g., \cite[p. 112]{Morgan - GMT}), and combining the three inequalities
we see that for a new constant the desired inequality holds. \end{proof}
\begin{thm}
\label{thm:Polygon-r^p}Let $B$ be a planar polygon containing the
origin and let $p>0$. Let $M$ be:
\begin{itemize}
\item The plane with density $r^{p}$ if the origin is contained in the
interior of $B$
\item The half plane with density $r^{p}$ if the origin is contained in
the boundary of $B$ but not at a vertex.
\item The $\theta_{0}$-sector with density $r^{p}$ if the origin is a
vertex of $B$ of angle $\theta_{0}$. 
\end{itemize}
Then for sufficiently small areas, isoperimetric regions in $B$ with
density $r^{p}$ are the same as those in $M$ in the natural sense.\end{thm}
\begin{proof}
Let $C_{1}$, and $C_{2}$ be circles (or semi-circles or circular
arcs, depending on $M$) about the origin of radii $r_{1}<r_{2}$
and such that the following hold:
\begin{enumerate}
\item $C_{2}$ (and therefore $C_{1}$) intersects the polygon only at the
side(s) containing the origin, or not at all if the origin is on the
interior of $B$.
\item The weighted length of any curve between $C_{1}$ and $C_{2}$ is
greater than $P_{1}$ where $P_{1}$ is the weighted perimeter of
$C_{1}$. 
\item $C_{2}$ contains less than half of the Euclidean area of $B$.
\end{enumerate}
Now, suppose we have a single closed curve inside $B$ (by which we
mean either closed in the traditional sense or intersecting the boundary
at both endpoints) of weighted length less than $P_{1}$. Then because
of $2$ in the above list, the curve must lie completely inside of
$C_{2}$ or completely outside of $C_{1}$. Because the density is
bounded away from 0 and $\infty$ outside of $C_{1}$, if we take
the length of the curve to be small enough that it cannot contain
all of $C_{1}$ on its side of smaller Euclidean area, then by applying
Lemma \ref{lem:Polygon-Ineq-With-Density} we find that for any such
curve completely outside of $C_{1}$, $P\geq c_{1}A^{1/2}$ for some
constant $c_{1}>0$ where $P$ and $A$ are the weighted perimeter
and weighted area of the region bounded by the curve with smaller
Euclidean area. On the other hand, we see by 3 in the above list that
any curve contained completely in $C_{2}$ bounds its region of smaller
Euclidean area completely inside $C_{2}$, and so by Prop. \ref{pro:more area/less perimeter argument works}
satifies $P\geq c_{2}A^{(p+1)/(p+2)}$ for some constant $c_{2}>0$
where $P$ and $A$ are as before. 

Now, $B$ has finite weighted area, and so by standard geometric measure
theory isoperimetric regions exist for all possible weighted areas.
So, let $R$ be an isoperimetric region of weighted area $A_{0}$
smaller than the weighted area of $C_{1}$, and let $P_{0}$ be the
weighted perimeter of a circle about the origin of weighted area $A_{0}$.
We will take $A_{0}$ to be small enough such that for any closed
curve in $B$ (again closed in the sense that it separates two regions
in $B$) with weighted length less than $P_{0}$ the region it bounds
with smaller weighted area also has smaller Euclidean area. Now, we
note that $R$ must have weighted perimeter less than or equal to
$P_{0}$ and thus less than $P_{1}$. 

Consider a single connected component of $R$. Suppose there is no
curve in its boundary such that the entire component is contained
in the region of smaller weighted area bounded by the curve. Then
the weighted area of the component is equal to the total weighted
area of $B$ minus the sum of all the smaller of the weighted areas
bounded by the curves forming the boundary. Because of the inequalities
we have proven above for weighted perimeter in terms of weighted area
in $B$, if we take $A_{0}$ to be small enough then any such component
must have weighted area larger than that of $A_{0}$.

So, any single connected component of $R$ is contained completely
in the region of smaller weighted area bounded by one of its boundary
curves, and thus is contained completely in $C_{2}$ or completely
in $P\backslash C_{1}$. Now, a circle about the origin of weighted
area $A$ has weighted perimeter $P$ such that $P=cA^{(p+1)/(p+2)}$
for some constant $c$ independent of the radius. Since any region
outside of $C_{1}$ satisfies $P\geq c_{1}A^{1/2}$, we see that for
$A\leq A_{0}$ and $A_{0}$ sufficiently small the circle about the
origin of the same weighted area will have strictly less weighted
perimeter than any component of weighted area $A$ lying outside $C_{1}$.
Thus, by replacing every component of $R$ that lies in $P\backslash C_{1}$
with a circle about the origin of the same weighted area, we obtain
a region $R'$ with multiplicity contained in $C_{2}$ having the
same weighted area as $R$ but less weighted perimeter than $R$.
This same region exists as a region with multiplicity in $M$, and
thus by Prop. \ref{pro: minimizers minimize in sector w/multiplicity},
has weighted perimeter greater than or equal to an isoperimetric region
with the same weighted area in $M$. Since such a region will also
exist in $P$ (as long as we have chosen $A_{0}$ small enough) with
the same weighted area and perimeter, we see that such a region is
isoperimetric, and in fact that $R$ must have been such a region
in the first place (otherwise it would have had strictly larger weighted
perimeter). 
\end{proof}
Using Proposition \ref{pro:Power change of coordinates} and our results
in the sector (Thm. \ref{thm:Major theorem on MWD} and the remark
following it about $p<0$) we obtain the following proposition:
\begin{prop}
\label{pro:MainTheoremPlanePerimeterDensity}In the plane with perimeter
density $r^{k}$, $k>-1$, and area density $r^{m}$ the following
are isoperimetric curves:
\begin{enumerate}
\item for $m\in(-\infty,-2]\cup(2k,\infty)$ there are none.
\item for $k\in(-1,0)$ and $m\in(-2,2k)$ the circle about the origin. 
\item for $k\in[0,\infty)$ and $m\in(-2,k-1]$ the circle about the origin.
\item for $k\in[0,\infty)$ and $m\in[k,2k]$ pinched circles through the
origin.
\end{enumerate}
\end{prop}
We also obtain a conjecture on the area density range between $k-1$
and $k$, which is missed by Proposition \ref{pro:MainTheoremPlanePerimeterDensity},
by doing the same analysis with Proposition \ref{pro:Power change of coordinates}
and the values for $\theta_{1}$ and $\theta_{2}$ in Conjecture \ref{con:Conjecture on major thm for MWD}.
\begin{conjecture}
\label{con:Conjecture in plane with different densities} In the plane
with perimeter density $r^{k}$, $k>-1$, and area density $r^{m}$
the following are isoperimetric curves:
\begin{enumerate}
\item for $k\in[0,\infty)$ and $m\in(k-1,k-1+\frac{1}{k+1}]$ the circle
about the origin.
\item for $k\in[0,\infty)$ and $m\in(k-1+\frac{1}{k+1},k-1+\frac{k+1}{2k+1})$
undularies.
\item for $k\in[0,\infty)$ and $m\in[k-1+\frac{k+1}{2k+1},k]$ pinched
circles through the origin.
\end{enumerate}
\end{conjecture}
\begin{rem*}
The circular arc being isoperimetric up to the $\pi/(p/2+1)$ sector
is equivalent to the circle being isoperimetric in the Euclidean plane
with any perimeter density $r^{p}$, $p\in[0,1]$ and Euclidean area.
As $p/2+1$ is the tangent line to $\sqrt{p+1}$ at 0, this is the
best possible bound we could obtain which is linear in the denominator. 
\end{rem*}
We now consider a more analytic formulation of the isoperimetric problem
in the $\theta_{0}$-sector with density $r^{p}$, and give an integral
inequality that is equivalent to proving the conjectured angle of
$\theta_{1}$. 
\begin{prop}
\label{prop:substituting alpha in isoperimetric inequality}In the
$\theta_{0}$-sector with density $r^{p}$, $p>0$, circles about
the origin are isoperimetric if and only if the inequality

\[
\left[\int_{0}^{1}r^{\frac{p+2}{p+1}}d\alpha\right]^{\frac{p+1}{p+2}}\leq\int_{0}^{1}\sqrt{r^{2}+\frac{r'^{2}}{[(p+1)\theta_{0}]^{2}}}d\alpha\]
holds for all $C^{1}$ functions $r(\alpha)$.\end{prop}
\begin{proof}
By Proposition \ref{pro:.isoperimetric curves are polar graphs containing origin},
an isoperimetric curve is a polar graph. The rest follows from manipulating
the integral formulas for weighted area and perimeter for polar graphs. \end{proof}
\begin{rem*}
\label{rem:Analytic proof of circles to theta<pi/p+1}This gives a
nice analytic proof that circles about the origin are isoperimetric
for $\theta_{0}=\pi/(p+1)$; letting $\theta_{0}=\pi/(p+1)$, we have\[
\left[\int_{0}^{1}r^{\frac{p+2}{p+1}}d\alpha\right]^{\frac{p+1}{p+2}}\leq\int_{0}^{1}\sqrt{r^{2}+\frac{r'^{2}}{\pi^{2}}}d\alpha.\]
When $p=0$, this corresponds to the isoperimetric inequality in the
half-plane with density $1$. As pointed out by Leonard Schulman of
CalTech, the left hand side is nonincreasing as a function of $p$,
meaning the inequality holds for all $p>0$.\end{rem*}
\begin{cor}
\label{cor:L^p norm inequality}In the $\theta_{0}$-sector with density
$r^{p}$, $p>0$, circles about the origin are isoperimetric for $\theta_{0}=\pi/\sqrt{p+1}$
if and only if the inequality\[
\left[\int_{0}^{1}r^{q}d\alpha\right]^{1/q}\leq\int_{0}^{1}\sqrt{r^{2}+\left(q-1\right)\frac{r'^{2}}{\pi^{2}}}d\alpha\]
holds for all $C^{1}$ functions $r(\alpha)$ for $1<q\leq2$.\end{cor}
\begin{proof}
In the inequality from Corollary \ref{prop:substituting alpha in isoperimetric inequality},
let $q=(p+2)/(p+1)$, and let $\theta_{0}=\pi/\sqrt{p+1}$.\end{proof}
\begin{rem*}
We wonder if an interpolation argument might work here. When $q=1$,
the result holds trivially (equality for all functions $r$), and
when $q=2$, the inequality follows from the isoperimetric inequality
in the half plane.
\end{rem*}

\section{\label{sec:CGCC-1}Constant Generalized Curvature Curves}

We look at constant generalized curvature curves in greater depth.
Theorem \ref{thm:classification of CGCC} classifies constant generalized
curvature curves. Proposition \ref{pro:If period of CGCC increasing, then...-1}
proves that if the half period of constant generalized curvature curves
is bounded above by $\pi(p+2)/(2p+2)$ and below by $\pi/\sqrt{p+1}$,
then Conjecture \ref{con:Conjecture on major thm for MWD} holds.
Our major tools for studying constant generalized curvature curves
are the second order constant generalized curvature equation and its
first integral. 

\begin{figure}
\begin{tabular}{ccc}
\begin{tabular}{c}
\includegraphics[width=1.4in,height=1.4in]{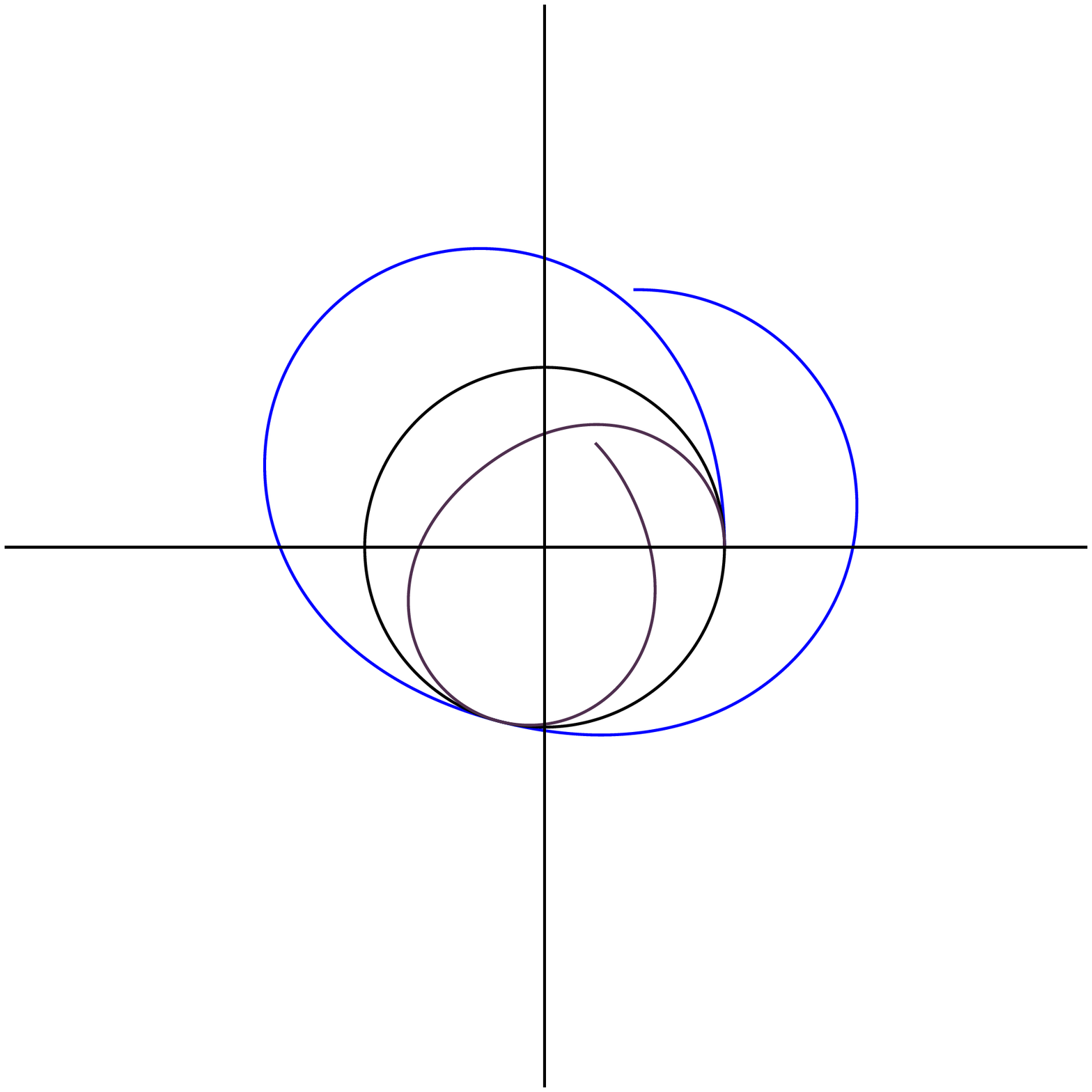}\tabularnewline
(a)\tabularnewline
\end{tabular} & \begin{tabular}{c}
\includegraphics[width=1.4in,height=1.4in]{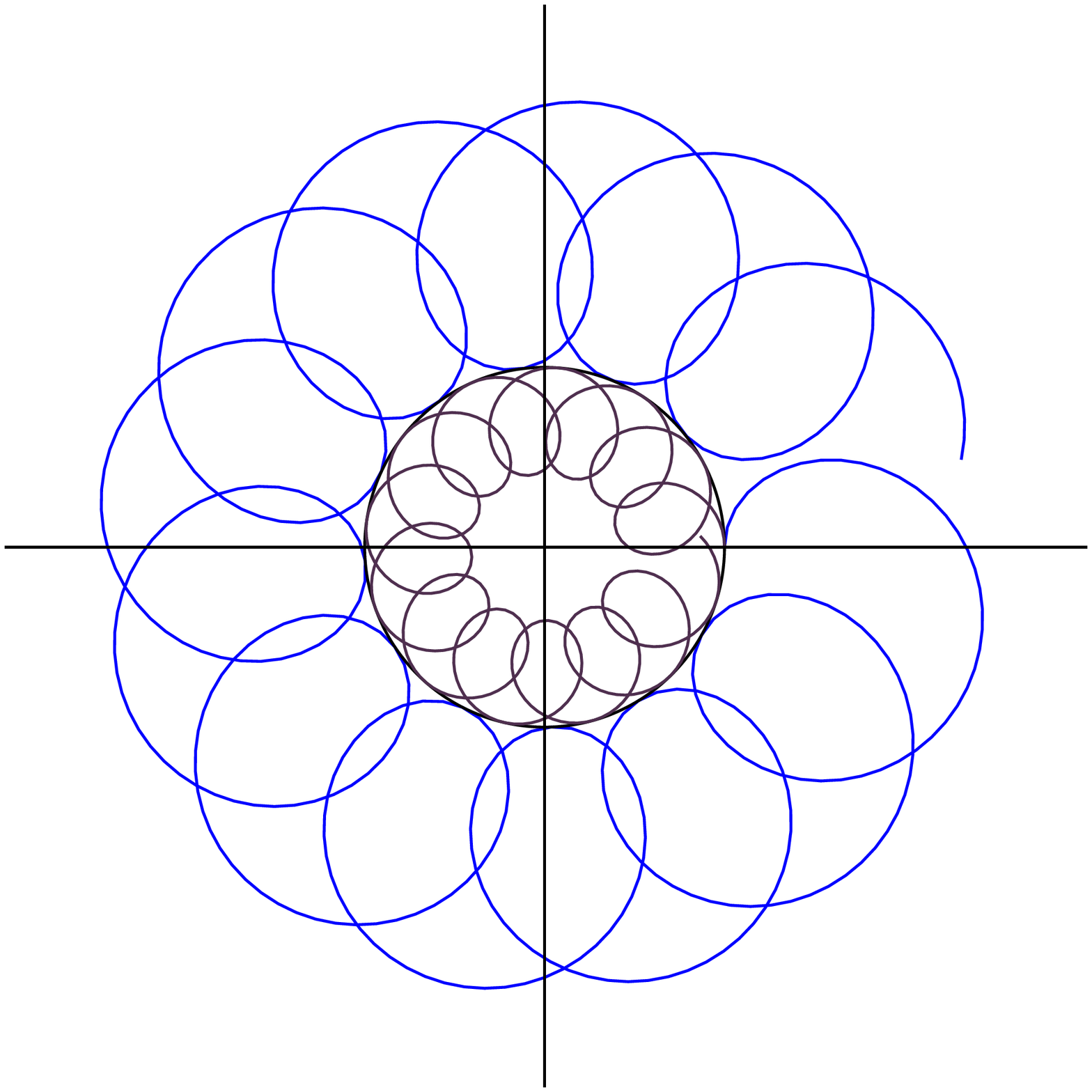}\tabularnewline
(b)\tabularnewline
\end{tabular}  & \begin{tabular}{c}
\includegraphics[width=1.4in,height=1.4in]{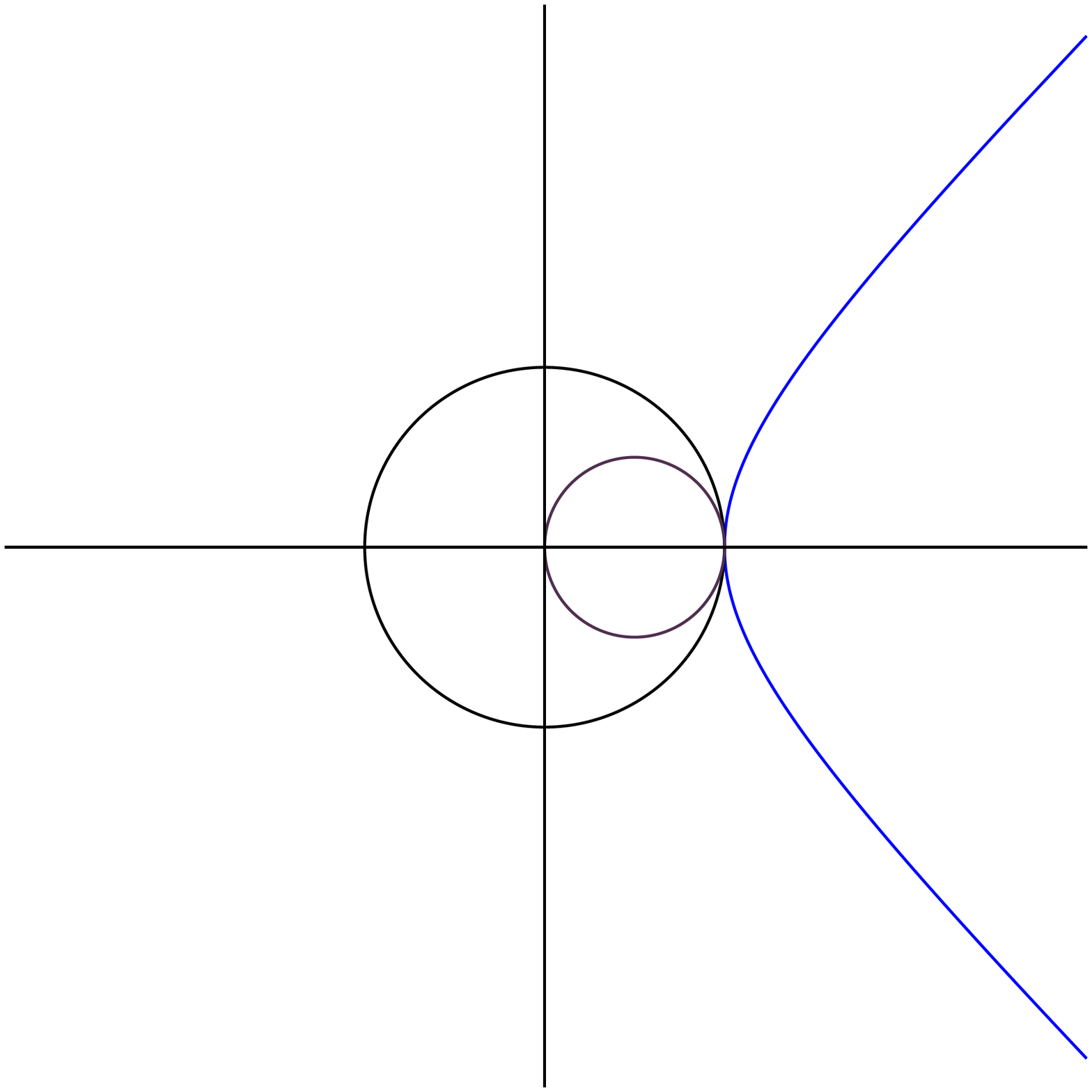}\tabularnewline
(c)\tabularnewline
\end{tabular}\tabularnewline
\end{tabular}

\caption{\label{fig:Possible-types-of CGCC}Types of constant generalized curvature
curves: a) for $0<\lambda<p+2$ nonconstant periodic polar graphs
(undularies), b) for $\lambda<0$ or $\lambda>p+2$ periodic nodoids,
c) for $\lambda=p+2$, a circle through the origin, for $\lambda=p+1$
a circle about the origin, for $\lambda=0$ a curve asymptotically
approaching the radial lines $\theta=\pm\pi/(2p+2)$.}

\end{figure}

\begin{thm}
\label{thm:classification of CGCC}In the plane with density $r^{p}$,
a curve with constant generalized curvature $\lambda$ normal at $(1,0)$
is (see Fig. \ref{fig:Possible-types-of CGCC}):
\begin{enumerate}
\item for $\lambda\in\left(0,p+2\right)-\{p+1\}$, a periodic undulary,
\item for $\lambda\notin\left[0,p+2\right]$, a periodic nodoid,
\item for $\lambda=p+2$, a circle through the origin,
\item for $\lambda=p+1$, a circle about the origin,
\item for $\lambda=0$, the geodesic \[
r(\theta)=\left(\sec((p+1)\theta)\right)^{1/(p+1)}\]
which asymptotically approaches the radial line $\theta=\pm\pi/(2p+2)$.
\end{enumerate}
\end{thm}
\begin{proof}
These results follow from the differential equations for constant
generalized curvature curves. Two points deserve special mention.
First, reflection and scaling provides a correspondence between curves
outside the circle and curves inside the circle. Second, to prove
a curve $\gamma$ which is not the circle, semicircle or geodesic
is periodic, it suffices to show $\gamma$ has a critical point after
$(1,0)$. It follows directly from the differential equation that
there is a radius at which $\gamma$ would attain a critical point.
To prove $\gamma$ actually attains this radius, we note that otherwise
$\gamma$ would have to asymptotically spiral to this radius. This
implies $\gamma$ has the same generalized curvature as a circle of
this radius, which the equations confirm cannot happen.
\end{proof}
Isoperimetric curves in the $\theta_{0}$-sector are polar graphs
with constant generalized curvature that intersect the boundary perpendicularly.
Given these restrictions and the classification of Theorem \ref{thm:classification of CGCC},
we can give a sufficient condition for Conjecture \ref{con:Conjecture on major thm for MWD}
on the value of the transitional angles.
\begin{prop}
\label{pro:If period of CGCC increasing, then...-1}In the plane with
density $r^{p}$, assume the half period of constant generalized curvature
curves normal at $(1,0)$ with generalized curvature $0<\lambda<p+1$
is bounded below by $\pi/\sqrt{p+1}$ and above by $\pi(p+2)/(2p+2)$.
Then the conclusions of Conjecture \ref{con:Conjecture on major thm for MWD}
hold.\end{prop}
\begin{proof}
Given the correspondence mentioned in the proof of Theorem \ref{thm:classification of CGCC},
it suffices to only check curves with $0<\lambda<p+1$. By Proposition
\ref{lem:Minimizers are circles, semicircles, or unduloids}, isoperimetric
curves are either circles about the origin, semicircles through the
origin, or undularies. An undulary is only in equilibrium when the
sector angle is equal to its half period. If there are no undularies
with half period less than $\pi/\sqrt{p+1}$ or greater than $\pi(p+2)/(2p+2)$,
the only possible isoperimetric curves outside of that range are the
circle and the semicircle. The proposition follows by the bounds given
in Theorem \ref{thm:Major theorem on MWD}.\end{proof}
\begin{rem}
\label{rem:CGCC remark}We now discuss some results on the periods
of constant generalized curvature curves, and provide numeric evidence.
Using the second order constant generalized curvature equation, one
can prove that curves with generalized curvature near that of the
circle have periods near $\pi/\sqrt{p+1}$ . Similarly, using the
first integral of the constant generalized curvature equation for
$dr/d\theta$, we see that the half period of a curve with constant
generalized curvature is given by\[
T=\int_{1}^{r_{1}}\frac{dr}{r\sqrt{\frac{r^{2p+2}}{\left(\frac{1-r_{1}^{p+1}}{1-r_{1}^{p+2}}\left(1-r^{p+2}\right)-1\right)^{2}}-1}}\]
where $r_{1}$ is the curve's maximum radius. Using this formula,
we generated numeric evidence for the bounds given in Proposition
\ref{pro:If period of CGCC increasing, then...-1}, as seen in Figure
\ref{fig:Numerical-evidence-for period of CGCC-1}. This plot is given
for $p=2$, and is representative of all $p$. Moreover, the integral
seems to be monotonic in $r_{1}$ which would imply that every undulary
with $0<\lambda<p+2$ is isoperimetric for exactly one $\theta_{0}$-sector.
Francisco L$\acute{\text{o}}$pez has suggested studying the integral
above with the techniques of complex analysis.
\end{rem}
\begin{figure}
\includegraphics[scale=0.75]{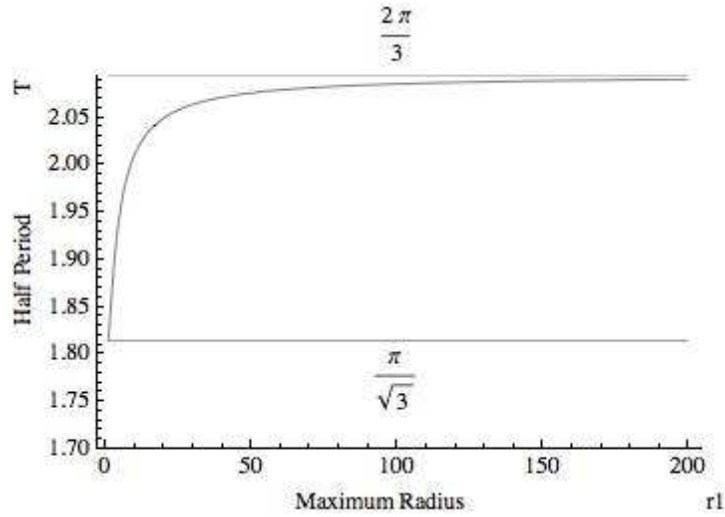}\caption{\label{fig:Numerical-evidence-for period of CGCC-1}This Mathematica
plot of undulary half period $T$ as a function of maximum radius
$r_{1}$ gives strong evidence for our main Conjecture \ref{con:Conjecture on major thm for MWD}
(see Prop. \ref{pro:If period of CGCC increasing, then...-1}).}

\end{figure}

\section{\label{sec:Ball Density}The Isoperimetric Problem in Sectors with
Disk Density}

In this section we classify the isoperimetric curves in the $\theta_{0}$-sector
with density 1 outside the unit disk $D$ centered at the origin and
$a>1$ inside $D$. Ca$\tilde{\text{n}}$ete \emph{et al.} \cite[Sect. 3.3]{Canete et al - Some isoperimetric problems in planes with density}
consider this problem in the plane, which is equivalent to the $\pi$-sector.
Proposition \ref{pro:anyarea-isosets-are} gives the potential isoperimetric
candidates. Theorems \ref{thm:thetalessthanpi}, \ref{thm:Pi<thetanot<aPI},
\ref{thm:aPI<thetanot} classify the isoperimetric curves for every
area and sector angle.
\begin{defn}
A \emph{bite }is an arc of $\partial D$ and another internal arc
(inside $D$), the angle between them equal to $\arccos(1/a)$ (see
Fig. \ref{fig:Isoperimetric-sets-for-thetanot-sector}(c)).
\end{defn}
\begin{figure}[h]
$\qquad$(a)\includegraphics{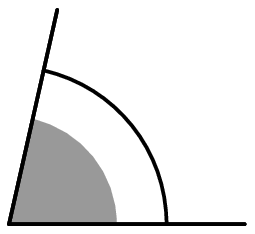} \includegraphics{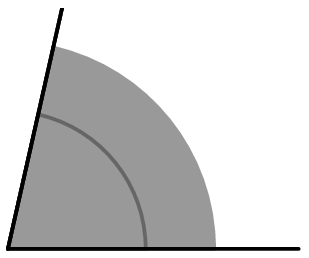}

$\qquad$(b)\includegraphics{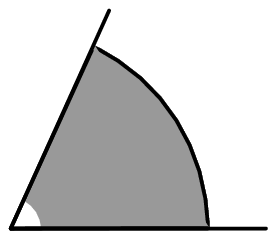} (c)\includegraphics{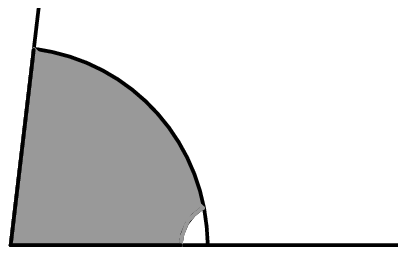}

$\qquad$(d)\includegraphics{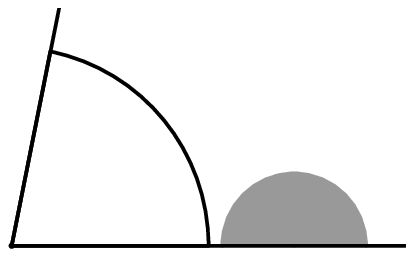} (e) \includegraphics{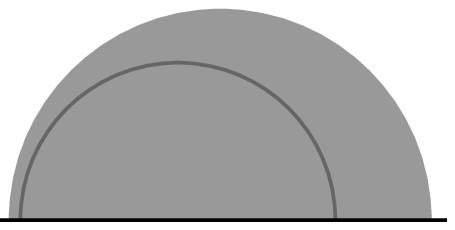}

\caption{\label{fig:Isoperimetric-sets-for-thetanot-sector}Isoperimetric sets
in sectors with disk density: (a) an arc about the origin inside or
outside $D$, (b) an annulus inside $D$, (c) a bite, (d) a semicircle
on the edge disjoint from the interior of $D$, (e) a semicircle centered
on the x-axis enclosing $D$ for $\theta_{0}=\pi$.}

\end{figure}

\begin{prop}
\label{pro:anyarea-isosets-are}In the $\theta_{0}$-sector with density
$a>1$ inside the unit disk $D$ and $1$ outside, for area $A>0$,
an isoperimetric set is one of the following (see Fig. \ref{fig:Isoperimetric-sets-for-thetanot-sector}):

(a) an arc about the origin inside or outside $D$;

(b) an annulus inside $D$ with $\partial D$ as a boundary;

(c) a bite;

(d) a semicircle on the edge of the sector disjoint from the interior
of $D$;

(e) a semicircle centered on the x-axis enclosing $D$ for $\theta_{0}=\pi$.\end{prop}
\begin{proof}
Any component of an isoperimetric curve has to meet the sector edge
perpendicularly since any component that meets the boundary at a different
angle contradicts Proposition \ref{des:Proposition - isoperimetric curves meet boundary perpendicular-1},
and any component that does not meet the sector edge can be rotated
about the origin (which will still give an isoperimetric region) until
it hits the sector edge at which point, since by Proposition \ref{pro:Isoperimetric exist in unit ball problem}
the curve is smooth, the intersection will be tangent so it will then
contradict Proposition \ref{des:Proposition - isoperimetric curves meet boundary perpendicular-1}.
Since each part of the boundary has to have constant generalized curvature
it must be made up of circular arcs (here we mean any circular arcs,
not necessarily arcs of circles about the origin as earlier in the
paper). We can discard the possibility of combinations of circular
arcs within the same density since one circular arc is better than
$n$ circular arcs (just as in the proof of Lemma \ref{lem:Minimizers are circles, semicircles, or unduloids}).
Therefore, there are five possibles cases: 

$\,$
\begin{enumerate}
\item A circular arc from one boundary edge to itself (including possibly
the origin).

\noindent There are three possibilities according to whether the semicircle
has 0, 1, or 2 endpoints inside the interior of $D$. A semicircle
with two endpoints inside $D$ has an isoperimetric ratio of $2\pi a$.
For a semicircle with one endpoint inside $D$ (Fig. \ref{fig:Semicircle-meeting-DB-perpendicular}),
by Proposition \ref{pro:Isoperimetric exist in unit ball problem}
Snell's Law holds. Then the only possible curve that intersects the
boundary normally would also have to intersect $D$ normally. Its
perimeter and area satisfy:%
\begin{figure}[h]
\includegraphics[scale=0.5]{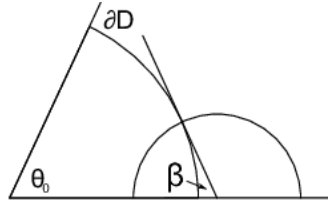}

\caption{\label{fig:Semicircle-meeting-DB-perpendicular}A semicircle meeting
$\partial D$ perpendicularly has more perimeter than a semicircle
disjoint from the interior of $D$ (Fig. \ref{fig:Isoperimetric-sets-for-thetanot-sector}d).}

\end{figure}

\[
P=r(\pi-\beta+a\beta),\]
\[
A=\frac{r^{2}}{2}\left(\pi-\beta+a\beta\right),\]
\[
\frac{P^{2}}{A}=2\left(\pi+\beta\left(a-1\right)\right).\]

\noindent Therefore a semicircle (d) outside $D$ with isoperimetric
ratio $2\pi$ is the only possibility.

$\,$

\item A circular arc from one boundary edge to another.

\noindent An arc (a) or annulus (b) about the origin are the only
possibilities for any $\theta_{0}<\pi$ or $A\leq a\theta_{0}/2$.
At $\theta_{0}=\pi$ and area $A>a\theta_{0}/2$ a semicircle (e)
centered on the x-axis enclosing $D$ is equivalent to a semicircle
centered at the origin. For $\theta_{0}>\pi$ and $A>a\theta_{0}/2$,
a semicircle tangent to $\partial D$ together with the rest of $\partial D$
(Fig. \ref{fig:Combinations of arcs}) is in equilibrium, but we will
show that it is not isoperimetric. Its perimeter and area satisfy:

\[
P=\pi R+(\theta_{0}-\pi),\qquad\qquad A=\frac{\pi}{2}\left(R^{2}-1\right)+\frac{\theta_{0}a}{2}.\]

\begin{figure}
\includegraphics[scale=0.5]{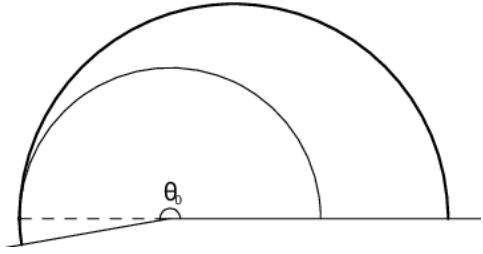}

\caption{\label{fig:Combinations of arcs} A semicircle tangent to $\partial D$
together with the rest of $\partial D$ is never isoperimetric.}

\end{figure}

\noindent Comparing it with an arc (a) about the origin we see that
it is not isoperimetric.

$\,$

\item Two circular arcs (c) meeting along $\partial D$ according to Snell's
Law (Prop. \ref{pro:Isoperimetric exist in unit ball problem}).

$\,$

\item Three or more circular arcs meeting along $\partial D$.

\noindent By Cañete \emph{et al.} \cite[Prop. 3.19]{Canete et al - Some isoperimetric problems in planes with density}
this is never isoperimetric.

$\,$

\item Infinitely many circular arcs meeting along $\partial D$.

\noindent By Cañete \emph{et al.} \cite[Prop. 3.19]{Canete et al - Some isoperimetric problems in planes with density}
this is never isoperimetric.

\end{enumerate}
\end{proof}
\begin{thm}
\label{thm:thetalessthanpi}For some $\theta_{2}<\pi$, in the $\theta_{0}$-sector
with density $a>1$ inside the unit disk $D$ and 1 outside, for $\theta_{0}\leq\pi$,
there exists $0<A_{0}<A_{1}<a\theta_{0}/2$, such that an isoperimetric
curve for area $A$ is (see Fig. \ref{fig:Isoperimetric-sets-for-thetanot-sector}):

(1) if $0<A<A_{0}$ , an arc about the origin if $\theta_{0}<\pi/a$,
semicircles on the edge disjoint from the interior of $D$ if $\theta_{0}>\pi/a$,
and both if $\theta_{0}=\pi/a$;

(2) if $A=A_{0}$, both type (1) and (3);

(3) if $A_{0}\leq A<A_{1}$ a bite, if $A_{1}<A<a\theta_{0}/2$ an
annulus inside $D$, and if $A=A_{1}$ both; if $\theta_{0}>\theta_{2}$,
$A_{1}=a\theta_{0}/2$ and the annulus is never isoperimetric;

(4) if $A\geq a\theta_{0}/2$, an arc about the origin; at $\theta_{0}=\pi$
any semicircle centered on the x-axis enclosing $D$.
\end{thm}

\begin{thm}
\label{thm:Pi<thetanot<aPI}In the $\theta_{0}$-sector with density
$a>1$ inside the unit disk $D$ and 1 outside, for $\pi<\theta_{0}\leq a\pi$,
there exists $A_{0}$, $A_{1}$, such that an isoperimetric curve
for area $A$ is (see Fig. \ref{fig:Isoperimetric-sets-for-thetanot-sector}):

(1) if $0<A<A_{0}$ , a semicircle on the edge disjoint from the interior
of $D$;

(2) if $A=A_{0}$, both type (1) and (3);

(3) if $A_{0}<A<a\theta_{0}/2$, a bite;

(4) if $a\theta_{0}/2\leq A<\theta_{0}^{2}(a-1)/2(\theta_{0}-\pi)$,
an arc about the origin;

(5) if $A=\theta_{0}^{2}(a-1)/2(\theta_{0}-\pi)$, both type (4) and
(6);

(6) if $A>\theta_{0}^{2}(a-1)/2(\theta_{0}-\pi)$, a semicircle on
the edge disjoint from the interior of $D$.
\end{thm}

\begin{thm}
\label{thm:aPI<thetanot}In the $\theta_{0}$-sector with density
$a>1$ inside the unit disk $D$ and 1 outside, for $\theta_{0}>a\pi$,
the isoperimetric curves for area $A$ are semicircles on the edge
disjoint from the interior of $D$ (Fig. \ref{fig:Isoperimetric-sets-for-thetanot-sector}e).
\end{thm}
\begin{figure}[H]
\includegraphics[scale=0.7]{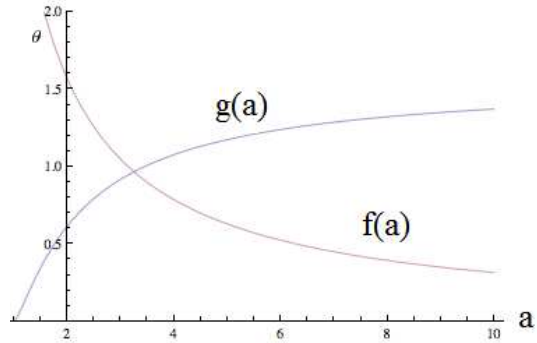}

\caption{\label{fig:Transition angle in function of density-1}For $\theta_{0}<f$
an arc about the origin is isoperimetric, while for $\theta_{0}>f$
a semicircle in the edge is isoperimetric. For area infinitesimally
less than $a\theta_{0}/2$, for $\theta_{0}<g$ an annulus inside
$D$ is isoperimetric, while for $\theta_{0}>g$ a bite is isoperimetric.}

\end{figure}

Given Proposition \ref{pro:anyarea-isosets-are}, most of the proofs
of Theorems \ref{thm:thetalessthanpi}, \ref{thm:Pi<thetanot<aPI},
\ref{thm:aPI<thetanot} are by simple perimeter comparison, which
we omit here. Figure \ref{fig:Transition angle in function of density-1}
shows numerically the transition between different isoperimetric curves.
The hardest comparison is handled by the following sample lemma.
\begin{lem}
\label{lem:Once bite wins, it always does}In the $\theta_{0}$-sector
with density $a>1$ inside the unit disk $D$ and 1 outside, for a
fixed angle and area $A<a\theta_{0}/2$, once a bite encloses more
area with less perimeter than an annulus inside $D$, it always does.\end{lem}
\begin{proof}
Increasing the difference $a\theta_{0}/2-A$, it will be sufficient
to prove that a bite is better than a scaling of a smaller bite because
an annulus changes by scaling. Therefore, there will be at most one
transition.

\begin{figure}[H]
\includegraphics[scale=0.5]{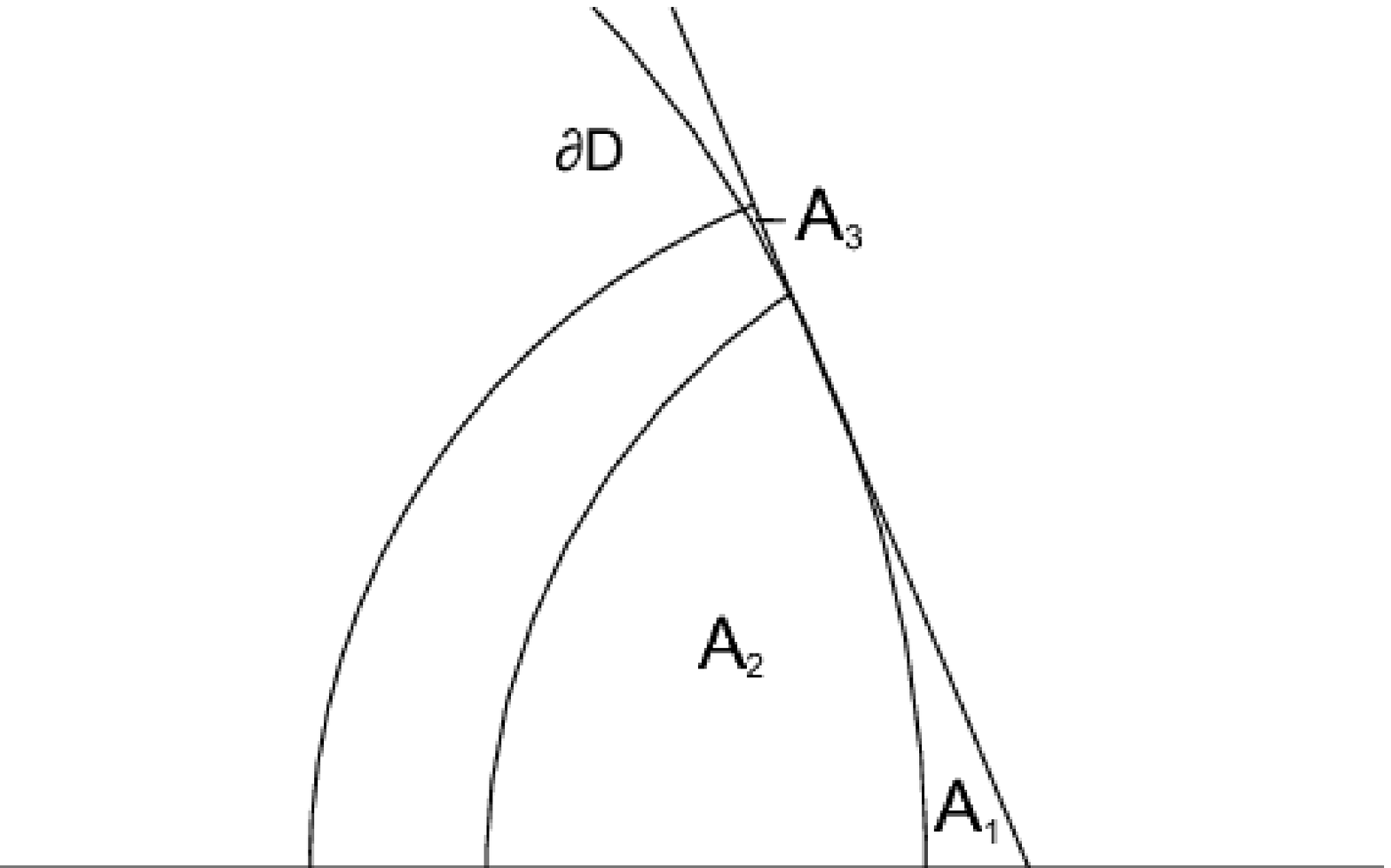}

\caption{\label{fig:Bite better than Scaling bite}A bite is better than a
scaling of a smaller bite, and hence once better than an annulus,
it is always better.}

\end{figure}

To prove that a bite does better than a scaling of a smaller bite
we show that it has less perimeter and more area. Take a smaller bite
and scale it by $1+\epsilon$ (see Fig. \ref{fig:Bite better than Scaling bite})
and eliminate the perimeter outside $D$. Another bite is created
with perimeter $P_{a}$ and area $A_{a}$:

\[
P_{a}<(1+\epsilon)P,\qquad A_{a}=\left(1+\epsilon\right)^{2}\left(A_{1}+A_{2}\right)-A_{3}>(1+\epsilon)^{2}A_{1},\]

\noindent Thus it is better than scaling a smaller bite.
\end{proof}

\section{Isoperimetric Problems in \label{sec:R^n w/ radial}$\mathbb{R}^{n}$
with Radial Density}

In this section we look at the isoperimetric problem in $\mathbb{R}^{n}$
with a radial density. We use symmetrization to show that if an isoperimetric
region exists, then an isoperimetric region of revolution exists (Lem.
\ref{lem:Isoperimetric regions bounded by surface of revolution}),
thus reducing the problem to a planar problem (Lem. \ref{lem:R^n w/density f(r) equivalent to 1/2 plane with density y^(n-2)f(r)}). 

We also consider the specific case of $\mathbb{R}^{n}$ with density
$r^{p}$. We first give a nonexistence result (Prop. \ref{pro:Minimizers don't exist for -n<p<0 in R^n})
for $-n\leq p<0$. The methods used by Carroll \emph{et al. }to study
planar densities $r^{-p},\; p<-2$ \cite[Prop. 4.3]{Carroll et al - Isoperimetric problem on planes with density}
can be used to recover an earlier result of Betta \emph{et al. }on
perimeter densities in $\mathbb{R}^{n}$ \cite[Thms. 2.1, 4.2, 4.3]{Betta et al},
and we provide a proof along these lines in Theorem \ref{thm:BettaEtAl-R^n}.
We then adapt this proof in Proposition \ref{Pro:Hypersphere-Minimize-p_LT_-n}
to show that hyperspheres about the origin bounding volume away from
the origin are uniquely isoperimetric in $\mathbb{R}^{n}$ with density
$r^{p},\; p<-n$. We conclude by conjecturing that hyperspheres through
the origin are isoperimetric in $\mathbb{R}^{n}$ with density $r^{p},\; p>0$.
(Conj. \ref{con:R^n with density r^p conjecture}). We'd like to thank
Bruno Volzone and Alexander Kolesnikov for bringing the paper by Betta\emph{
et al. }to our attention, as well as Robin Walters and Frank Morgan
for providing initial versions of our averaging proof of Theorem \ref{thm:BettaEtAl-R^n}.
\begin{lem}
\label{lem:Isoperimetric regions bounded by surface of revolution}In
$\mathbb{R}^{n}$ with a radial density, if there exists an isoperimetric
region, then there exists an isoperimetric region of revolution.\end{lem}
\begin{proof}
Using spherical symmetrization \cite[Sect. 9.2]{Burago and Zalgaller- Geometric Inequalities}
generalized with density \cite{Symmetrization Paper}, we can replace
the intersection of the isoperimetric region with each sphere about
the origin by a polar cap of the same volume.
\end{proof}

\begin{lem}
\label{lem:R^n w/density f(r) equivalent to 1/2 plane with density y^(n-2)f(r)}The
isoperimetric problem in $\mathbb{R}^{n}$ with density $f(r)$ is
equivalent to the isoperimetric problem in the half plane $y>0$ with
density $y^{n-2}f(r)$.\end{lem}
\begin{proof}
Up to a constant, this is the quotient space of $\mathbb{R}^{n}$
with density $f(r)$ modulo rotations about an axis. By Lemma \ref{lem:Isoperimetric regions bounded by surface of revolution},
there exists an isoperimetric region bounded by surfaces of revolution,
so a symmetric isoperimetric region in $\mathbb{R}^{n}$ corresponds
to an isoperimetric region in the quotient space and vice-versa.
\end{proof}
Isoperimetric curves in the plane with density $r^{p}$ are circles
about the origin when $p<-2$, do not exist when $-2\leq p<0$, and
are circles through the origin when $p>0$ (\cite[Props. 4.2, 4.3]{Carroll et al - Isoperimetric problem on planes with density},
\cite[Thm. 3.16]{Dahlberg et al - Isoperimetric regions in planes with density r^p}).
We now look at the isoperimetric problem in $\mathbb{R}^{n}$ with
density $r^{p}$.
\begin{prop}
\label{pro:Minimizers don't exist for -n<p<0 in R^n}For $-n\leq p<0$
in $\mathbb{R}^{n}$ with density $r^{p}$, isoperimetric regions
do not exist: you can enclose any volume with arbitrarily small perimeter.\end{prop}
\begin{proof}
Consider a sphere $S$ of radius $R>0$ not containing the origin.
Let $r_{min},$ $r_{max}$ be the minimum and maximum values of $r$
attained on $S$, note that $0<r_{min}=r_{max}-2R$. Let $V,$ $P$
be the volume and perimeter respectively of $S$. We have that:\[
V>c_{0}R^{n}\min{}_{S}(r^{p})=c_{0}R^{n}r_{max}^{p}.\]
Working with a fixed $V$, we get that for any sphere of radius $R$
with volume $V$, \[
r_{max}>c_{1}R^{q}\]
 where $q=-n/p\geq1$ and $c_{1}=\left(V/c_{0}\right)^{1/p}$. So,
if we fix $V$ large enough so that $c_{1}>3$, we obtain \[
r_{min}>c_{1}R^{q}-2R>R^{q}\]
for any $R>1$. Looking at $P$ we have

\[
P<c_{2}R^{n-1}\max{}_{S}(r^{p})=c_{2}R^{n-1}r_{min}^{p}<c_{2}R^{n-1}(R^{q})^{p}=c_{2}R^{-1},\]
which can be made as small as we want by taking $R$ large, and so
since there are spheres of arbitrarily large radius not containing
the origin and having volume $V$, we see that there are regions of
arbitrarily small perimeter enclosing this fixed volume $V$. Then,
by scaling these regions, we obtain the result for all volumes. \end{proof}
\begin{rem*}
A proof that applies for $-n<p<-n+1$ was given by \cite[Prop. 4.2]{Carroll et al - Isoperimetric problem on planes with density}.
Their next proposition, \cite[Prop 4.3]{Carroll et al - Isoperimetric problem on planes with density},
which extends immediately to sectors, is by Proposition \ref{pro:Power change of coordinates}
equivalent to the $r^{p}$ case of a theorem of Betta \emph{et al.
}\cite[Thm. 4.3]{Betta et al}. Here we use the averaging technique
of Carroll \emph{et al.} to provide a short proof of the whole theorem
of Betta \emph{et al.}\end{rem*}
\begin{thm}[{\cite[Thm. 4.3]{Betta et al}}]
\label{thm:BettaEtAl-R^n}Suppose $a$ is a non-negative non-decreasing
smooth function on the positive real axis with $a(s)>0$ for $s>0$
and with \textup{\[
\left[a(s^{1/n})-a(0)\right]s^{1-1/n}\]
}convex. Then hyperspheres about the origin are uniquely isoperimetric
in $\mathbb{R}^{n}$ with Euclidean volume and surface area density
$a$. \end{thm}
\begin{proof}
As in Betta \emph{et al., }we can assume $a(0)$=0, since if we let
\[
a_{1}(s)=a(s)-a(0),\]
then $a$-weighted surface area is just the sum of $a_{1}$-weighted
surface area and a constant times Euclidean surface area, and so if
hyperspheres are isoperimetric for both these densities and uniquely
isoperimetric for one then they will also be uniquely isoperimetric
for $a$. 

Let $B$ be a region bounded by rectifiable hypersurfaces. We will
show that a ball about the origin $B'$ of the same volume has less
surface area than $B$ with equality only when $B$ is also a ball.
In fact, we will prove a stronger statement: Note that the surface
area of $B$ is greater than or equal to the tangential surface area
of $B$, i.e. the component of surface area perpendicular to radial
lines, with equality only when the boundary of $B$ consists only
of hypersurfaces. We will show that $B'$ has less tangential surface
area than $B$. 

In $\mathbb{R}^{n}$ with radial surface area density $a$ and unit
volume density, in polar coordinates the volume element is 

\[
dV=r^{n-1}\cdot drd\Theta,\]
and the tangential surface area element is \[
dQ=a(r)r^{n-1}d\Theta\]
where $d\Theta$ is the element of surface area on the unit sphere.
Under the change of coordinates $s=r^{n}$, the volume element becomes
\[
dV=\frac{1}{n}dsd\Theta\]
and the tangential surface area element becomes\[
dQ=fd\Theta\]
where $f(s)=a(s^{1/n})s^{1-1/n}$. Now, in the $s$-variable we obtain
\[
\int_{B}dV=\int_{\partial B}\pm s\cdot\frac{1}{n}d\Theta\]
where $s$ is signed according to the radial orientation of the boundary
(positive if pointing away from the origin, negative if pointing towards
the origin). Let $t(\Theta)$ be the sign-weighted sum of all the
$s$ values of points in the intersection of $\partial B$ and a radial
ray in the direction of $\Theta$. Note that this is not defined in
directions where $\partial B$ is radial, but this is a set of $(n-1)$-dimensional
Lebesgue measure $0$ in the $\Theta$-space so it will not effect
our computations. We then have that\[
\int_{B}dV=\int t(\Theta)\cdot\frac{1}{n}d\Theta,\]
where the latter integral is taken over all possible directions. Now,
let $t_{avg}$ be the average value of $t$ over all directions. Then
in these coordinates we see that $B'$, the ball about the origin
with the same volume as $B$, is actually the ball about the origin
of radius $t_{avg}$: \[
\int_{B}dV=\int t(\theta)\cdot\frac{1}{n}d\Theta=\int t_{avg}\cdot\frac{1}{n}d\Theta.\]
Now, we get

\begin{equation}
\int_{\partial B}dQ\geq\int f(t(\theta))\cdot d\Theta\geq\int f(t_{avg})\cdot d\Theta=\int_{\partial B'}dQ,\label{eq:IneqChainBettaProof}\end{equation}
the first inequality coming from $f$ increasing, and the second inequality
from $f$ convex (note that the middle integral should in general
be taken only over the directions where the intersection with $\partial B$
is non-empty, but in this case since $f$ is zero at the origin we
can in fact take it over all directions). So, we see that $B'$ has
tangential surface area less than or equal to that of $B$, and since
all of the surface area of $B'$ is counted in the tangential surface
area ($B'$ is a ball about the origin), we conclude that $B'$ has
surface area less than or equal to that of $B$ with equality possible
only if all of the surface area of $B$ is also tangential, that is
if $\partial B$ consists only of hyperspheres about the origin. But
if $\partial B$ consists of multiple hyperspheres about the origin
then the first inequality in \ref{eq:IneqChainBettaProof} must be
strict, and thus we conclude that equality holds if and only if $B=B'$. \end{proof}
\begin{prop}
\label{Pro:Hypersphere-Minimize-p_LT_-n}In $\mathbb{R}^{n}$ with
density $r^{p}$ for $p<-n$ , hyperspheres about the origin are uniquely
isoperimetric (bounding volume away from the origin).\end{prop}
\begin{proof}
We use the substitution $s=r^{p+n}$. This will give us tangential
surface area element\[
dQ=s^{1+\frac{-1}{p+n}}d\theta\]
and the rest of the proof of Theorem \ref{thm:BettaEtAl-R^n} will
carry through identically since \[
\frac{-1}{p+n}>0.\]
 \end{proof}
\begin{rem*}
The proof of Theorem \ref{thm:BettaEtAl-R^n} also gives this result
in the $\theta_{0}$-sector and even more generally in any cone over
a subset of $S^{n-1}$, but we may have to change the convexity condition:
if hyperspheres about the origin are not isoperimetric in such a space
with Euclidean surface area, then we cannot perform the trick of subtracting
off $a(0)$, and thus we will want to impose the stronger condition
instead that $a(s^{1/n})s^{1-1/n}$ is convex. We note that if we
take the cone over a convex subset of $S^{n-1}$ then in the Euclidean
case hyperspheres about the origin will be isoperimetric (\cite[Thm 1.1]{LionsEtAl-ConvexCones},
see also \cite[Rmk. after Thm. 10.6 ]{Morgan-RiemannianGeom}) and
therefore the original convexity condition will suffice. The proof
of Proposition \ref{Pro:Hypersphere-Minimize-p_LT_-n} goes through
in any of these spaces with no new conditions.

For $\mathbb{R}^{n}$ with density $r^{p},\; p>0$, we make a conjecture
analagous to the planar case which was solved by Dahlberg \emph{et
al. \cite{Dahlberg et al - Isoperimetric regions in planes with density r^p}:}\end{rem*}
\begin{conjecture}
\label{con:R^n with density r^p conjecture}In $\mathbb{R}^{n}$ with
density $r^{p},\; p>0$, hyperspheres through the origin are uniquely
isoperimetric. \end{conjecture}

$\quad$

\author{\noindent Alexander Díaz, Department of Mathematical Sciences, }

\author{\noindent University of Puerto Rico, Mayag$\ddot{\text{u}}$ez, PR
00681}

\emph{E-mail address}: alexander.diaz1@upr.edu

$\quad$

\author{\noindent Nate Harman, Department of Mathematics and Statistics,}

\author{\noindent University of Massachusetts, Amherst, MA 01003}

\emph{E-mail address}: nateharman1234@yahoo.com

$\quad$

\author{\noindent Sean Howe, Department of Mathematics,}

\author{\noindent Université Paris-Sud 11, Orsay, France 91405}

\emph{E-mail address}: seanpkh@gmail.com

$\quad$

\author{\noindent David Thompson, Department of Mathematics and Statistics,}

\author{\noindent Williams College, Williamstown, MA 01267}

\emph{E-mail address}: dat1@williams.edu

$\quad$

\author{\noindent Mailing Address: c/o Frank Morgan, Department of Mathematics
and Statistics, Williams College, Williamstown, MA 01267}

\emph{E-mail address}: Frank.Morgan@williams.edu

\begin{thebibliography}{RCBM}
\bibitem[ACDLV]{Adams et al - Isoperimetric Regions in Gauss Sectors}Adams,
Elizabeth; Corwin, Ivan; Davis, Diana; Lee, Michelle; Visocchi, Regina.
Isoperimetric regions in Gauss sectors, Rose-Hulman Und. Math. J.
8 (2007), no. 1.

\bibitem[BZ]{Burago and Zalgaller- Geometric Inequalities}Burago,
Yu. D.; Zalgaller, V.A. \emph{Geometric Inequalities, }Springer-Verlag,
1980.

\bibitem[BBMP]{Betta et al}Betta, M. F.; Brock, F.; Mercaldo, A;
Posteraro, M. R. A weighted isoperimetric inequality and applications
to aymmetrization, J. of Inequal. \& Appl., 1999, Vol. 4, 215-240.

\bibitem[CMV]{Canete et al - Some isoperimetric problems in planes with density}Cañete,
Antonio; Miranda, Michele, Jr.; Vittone, Davide. Some isoperimetric
problems in planes with density. J. Geom. Anal. 20 (2010), no. 2,
243-290.

\bibitem[CJQW]{Carroll et al - Isoperimetric problem on planes with density}Carroll,
Colin; Jacob, Adam; Quinn, Conor; Walters, Robin. The isoperimetric
problem on planes with density, Bull. Austral. Math. Soc. 78 (2008),
177-197. 

\bibitem[CHHSX]{Corwin et al - Differential geometry of manifolds with density}Corwin,
Ivan; Hoffman, Neil; Hurder, Stephanie; Sesum, Vojislav; Xu, Ya. Differential
geometry of manifolds with density, Rose-Hulman Und. Math. J. 7 (1)
(2006). 

\bibitem[DDNT]{Dahlberg et al - Isoperimetric regions in planes with density r^p}Dahlberg,
Jonathan; Dubbs, Alexander; Newkirk, Edward; Tran, Hung. Isoperimetric
regions in the plane with density $r^{p}$, New York J. Math. 16 (2010),
31-51. nyjm.albany.edu/j/2010/16-4v.pdf.

\bibitem[DHHT]{Report}Díaz, Alexander; Harman, Nate; Howe, Sean;
Thompson, David. SMALL Geometry Group 2009 Report, Williams College,
http://www.williams.edu/Mathematics/fmorgan/G09.pdf.

\bibitem[EMMP]{Engelstein et al - Isoperimetric problems on the sphere and on surfaces with density}Engelstein,
Max; Marcuccio, Anthony; Maurmann, Quinn; Pritchard, Taryn. Isoperimetric
problems on the sphere and on surfaces with density, New York J. Math.
15 (2009), 97-123.

\bibitem[LP]{LionsEtAl-ConvexCones}Lions, Pierre-Louis; Pacella,
Filomena. Isoperimetric inequalities for convex cones, Proc. Amer.
Math. Soc. 109, 477-485, (1990).

\bibitem[LB]{Lopez - The double bubble problem on the cone}Lopez,
Robert; Baker, Tracy Borawski. The double bubble problem on the cone,
New York J. Math. 12 (2006), 157-167.

\bibitem[MM]{Maurmann Morgan - Isoperimetric comparison theorems for manifolds with density}Maurmann,
Quinn; Morgan, Frank. Isoperimetric comparison theorems for manifolds
with density, Calc. Var. PDE 36 (2009), 1–5.

\bibitem[M1]{Morgan - GMT}Morgan, Frank. \emph{Geometric Measure
Theory: a Beginner’s Guide}, 4th ed., Academic Press, London, 2009.

\bibitem[M2]{Morgan-In Polytopes Small Balls about Some Vertex Minimize Perimeter }Morgan,
Frank. In polytopes, small balls about some vertex minimize perimeter,
Geom. Anal. 17 (2007), 97-106.

\bibitem[M3]{Morgan-isoperimetric balls in cones over tori.}Morgan,
Frank. Isoperimetric balls in cones over tori, Ann. Glob. Anal. Geom.
35 (2009), 133–137.

\bibitem[M4]{Morgan- Isoperimetric estimates on products}Morgan,
Frank. Isoperimetric estimates on products, Ann. Global Anal. Geom.
110 (2006), 73-79.

\bibitem[M5]{Morgan-Manifolds with Density}Morgan, Frank. Manifolds
with density, Notices AMS 52 (2005), 853-858.

\bibitem[M6]{Morgan - Regularity of isoperimetric hypersurfaces in Riemannian manifolds}Morgan,
Frank. Regularity of isoperimetric hypersurfaces in Riemannian manifolds,
Trans. AMS 355 (2003), 5041-5052.

\bibitem[M7]{Morgan-RiemannianGeom}Morgan, Frank. \emph{Riemannian
Geometry: A Beginner's Guide}, A. K. Peters, 1998.

\bibitem[MHH]{Symmetrization Paper}Morgan, Frank; Howe, Sean; Harman,
Nate. Steiner and Schwarz symmetrization in warped products and fiber
bundles with density, Revista Mat. Iberoamericana, to appear; arXiv.org
(2009).

\bibitem[MR]{Morgan et al. - isoperimetric regions in cones}Morgan,
Frank; Ritor$\acute{\text{e}}$, Manuel. Isoperimetric regions in
cones, Trans. AMS 354 (2002), 2327-2339.

\bibitem[O]{Osserman - the four or more vertex theorem} Osserman,
Robert. The four-or-more vertex theorem, Amer. Math. Monthly 92 (1985),
332-337.

\bibitem[PR]{Pedrosa Ritore-isoperimetric domains in the Riemannian Product}Pedrosa,
Renato; Ritoré, Manuel. Isoperimetric domains in the Riemannian product
of a circle with a simply connected space form and applications to
free boundary problems, Indiana U. Math. J. 48 (1999), 1357–1394.

\bibitem[RCBM]{Rosales et al - On the Isoperimetric Problem in Euclidean Space with Density}
Rosales, C$\acute{\text{e}}$sar; Cañete, Antonio; Bayle, Vincent;
Morgan, Frank. On the isoperimetric problem in Euclidean space with
density, Calc. Var. PDE 31 (2008), 27-46.

\end{thebibliography}
\end{document}